\documentclass[11pt, reqno]{amsart}

\usepackage{lmodern}
\usepackage{bbm}

\usepackage{amsmath, amssymb, amsfonts, amsthm}
\usepackage{mathtools, mathdots}
\usepackage{array}
\usepackage{stackengine, stmaryrd}
\usepackage{xparse} 

\usepackage[margin=1in]{geometry}
\usepackage{float}
\usepackage{graphicx}
\usepackage[font=smaller, labelfont=bf]{caption}
\usepackage{booktabs, multirow, tabularx}
\usepackage{wrapfig}
\usepackage{comment}
\usepackage[shortlabels]{enumitem}
\setlist[enumerate]{wide=15pt, leftmargin=15pt, labelwidth=15pt, align=left}

\numberwithin{equation}{section}
\newtheorem{theorem}[equation]{Theorem}

\newtheorem{proposition}[equation]{Proposition}
\newtheorem{lemma}[equation]{Lemma}
\newtheorem{corollary}[equation]{Corollary}

\theoremstyle{definition}
\newtheorem{rmk}[equation]{Remark}
\newenvironment{remark}[1][]{\begin{rmk}[#1] \pushQED{\qed}}{\end{rmk}}
\newtheorem{eg}[equation]{Example}
\newenvironment{example}[1][]{\begin{eg}[#1] \pushQED{\qed}}{\end{eg}}
\newtheorem{defn}[equation]{Definition}
\newenvironment{definition}[1][]{\begin{defn}[#1] \pushQED{\qed}}{\end{defn}}

\newtheorem{innercustomgeneric}{\customgenericname}
\providecommand{\customgenericname}{}
\newcommand{\newcustomtheorem}[2]{%
  \newenvironment{#1}[1]{%
    \renewcommand\customgenericname{#2}%
    \renewcommand\theinnercustomgeneric{##1}%
    \innercustomgeneric
  }{\endinnercustomgeneric}
}
\newcustomtheorem{customthm}{Theorem}
\newcustomtheorem{customprop}{Proposition}
\newcustomtheorem{customcor}{Corollary}

\usepackage{tikz, tikz-cd, xcolor}
\usetikzlibrary{
  arrows,
  calc,
  chains,
  fit,
  matrix,
  positioning,
  shapes,
  decorations.pathmorphing,
  decorations.pathreplacing,
  angles,
  quotes
}
\tikzcdset{scale cd/.style={
  every label/.append style={scale=#1},
  cells={nodes={scale=#1}}}}

\PassOptionsToPackage{hyphens}{url}
\usepackage[colorlinks=true, linkcolor=blue, citecolor=black, urlcolor=blue]{hyperref}

\newcommand{\Z}{\mathbb{Z}}

\newcommand{\bbA}{\mathbb{A}}
\newcommand{\DD}{\mathbb{D}}
\newcommand{\EE}{\mathbb{E}}
\newcommand{\xto}[1]{\xrightarrow{#1}}
\newcommand{\toto}{\rightrightarrows}

\newcommand{\ddim}{\underline{\dim}\,}

\DeclareDocumentCommand \ev { o o } {%
  \IfNoValueTF {#1}
    {\textnormal{ev}}
    {\textnormal{ev}(#1,#2)}
}

\newenvironment{nalign}{
  \begin{equation}
  \begin{aligned}
}{
  \end{aligned}
  \end{equation}
  \ignorespacesafterend
}

\newenvironment{nalign*}{
  \begin{equation*}
  \begin{aligned}
}{
  \end{aligned}
  \end{equation*}
  \ignorespacesafterend
}

\DeclareMathOperator{\rep}{rep}
\DeclareMathOperator{\Hom}{Hom}

\DeclareMathOperator{\coker}{coker}

\DeclareMathOperator{\AR}{AR}

\DeclareMathOperator{\im}{im}

\begin{document}

\title{Completely-decomposable subcategories of quiver representations}
\author[Y. Diaz]{Yariana Diaz}
\address{Department of Mathematics, Statistics, and Computer Science\\ 
Macalester College\\ 
Saint Paul, MN 55104}
\email{ydiaz@macalester.edu}

\begin{abstract}
    When filtering a topological space by a single parameter, the theory of quiver representations provides a complete framework for decomposing the resulting persistence module to obtain its barcode. This is achieved by interpreting the persistence module as a representation of a Type $\bbA_n$ quiver. The complexity increases significantly when filtering by two or more parameters. In particular, multi-parameter persistence typically yields tame or wild type quivers whose indecomposable representations are more complicated to describe and for which arbitrary representations are much more difficult to decompose. The theme of this work is to provide a framework for restricting to subcategories of quiver representations whose objects can be distinguished from one another through computational means.
\end{abstract}

\maketitle

\tableofcontents

\section{Introduction and Motivation}

This work is motivated by the question of decomposing persistence modules in the context of topological data analysis (TDA). In TDA, persistent homology is used to study the shape of data by tracking the appearance and disappearance of topological features across a filtration indexed by a poset. Then, applying homology with field coefficients yields a collection of vector spaces and linear maps between them. The resulting structure, called a persistence module, can be interpreted as a representation of a quiver $Q$ whose underlying diagram reflects the poset indexing the filtration. Other realizations of the structure of the persistence module include descriptions of persistence modules indexed over $\Z^n$ as $\Z^n$-graded modules over the polynomial ring of $n$ variables \cite{CZmultidim2007}. We can also view a persistence module as a functor from the indexing poset $(T,\leq)$, where $T \subseteq k^n$ for some field $k$, to the category $\text{Vec}_{k}$ of vector spaces over $k$.

Single-parameter persistence theory focuses on filtrations generated by a single parameter. The study of such filtrations is very well understood and corresponds very neatly to the quiver representation theory of type $\bbA_n$ quivers \cite{Oudot}. Type $\bbA_n$ quivers have the advantage that they have finitely many indecomposable representations. Quivers with this property are called \textbf{representation-finite}. By Gabriel's theorem in \cite{gabriel}, representation-finite quivers are exactly those which have an underlying graph described by the Dynkin diagrams, pictured below.

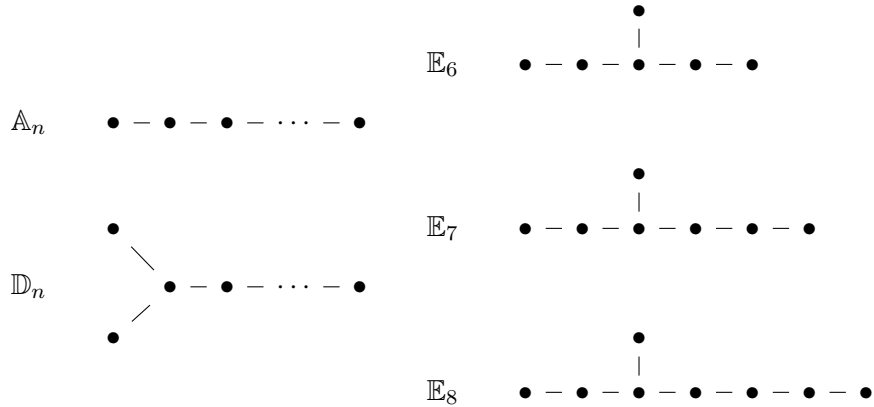
\begin{figure}[ht]
    \caption{Dynkin diagrams, types $\bbA_n, \DD_n, \EE_6, \EE_7, \EE_8$}
    \[\begin{tikzcd}[ampersand replacement=\&,sep=tiny]
	\&\&\&\&\&\&\&\&\&\&\&\& \bullet \\
	\&\&\&\&\&\&\&\& {\EE_6} \&\& \bullet \& \bullet \& \bullet \& \bullet \& \bullet \\
	{\bbA_n} \&\& \bullet \& \bullet \& \bullet \& \cdots \& \bullet \\
	\&\&\&\&\&\&\&\&\&\&\&\& \bullet \\
	\&\& \bullet \&\&\&\&\&\& {\EE_7} \&\& \bullet \& \bullet \& \bullet \& \bullet \& \bullet \& \bullet \\
	{\DD_n} \&\&\& \bullet \& \bullet \& \cdots \& \bullet \\
	\&\& \bullet \&\&\&\&\&\&\&\&\&\& \bullet \\
	\&\&\&\&\&\&\&\& {\EE_8} \&\& \bullet \& \bullet \& \bullet \& \bullet \& \bullet \& \bullet \& \bullet
	\arrow[no head, from=5-3, to=6-4]
	\arrow[no head, from=7-3, to=6-4]
	\arrow[no head, from=6-4, to=6-5]
	\arrow[no head, from=6-5, to=6-6]
	\arrow[no head, from=3-4, to=3-5]
	\arrow[no head, from=2-11, to=2-12]
	\arrow[no head, from=2-12, to=2-13]
	\arrow[no head, from=1-13, to=2-13]
	\arrow[no head, from=3-3, to=3-4]
	\arrow[no head, from=2-13, to=2-14]
	\arrow[no head, from=2-14, to=2-15]
	\arrow[no head, from=5-11, to=5-12]
	\arrow[no head, from=5-12, to=5-13]
	\arrow[no head, from=5-13, to=5-14]
	\arrow[no head, from=5-14, to=5-15]
	\arrow[no head, from=5-15, to=5-16]
	\arrow[no head, from=4-13, to=5-13]
	\arrow[no head, from=8-11, to=8-12]
	\arrow[no head, from=8-13, to=8-12]
	\arrow[no head, from=7-13, to=8-13]
	\arrow[no head, from=8-13, to=8-14]
	\arrow[no head, from=8-14, to=8-15]
	\arrow[no head, from=8-15, to=8-16]
	\arrow[no head, from=8-16, to=8-17]
	\arrow[no head, from=3-5, to=3-6]
	\arrow[no head, from=6-6, to=6-7]
	\arrow[no head, from=3-6, to=3-7]
\end{tikzcd}\]
\end{figure}

Moreover, by Gabriel's Theorem and the Krull-Remak-Schmidt Theorem, the representations of type $\bbA_n$ quiver (and the other Dynkin type quivers) can always be decomposed into a direct sum of indecomposable representations. This decomposition is unique up to isomorphism and permutation of the indecomposable summands. These indecomposable representations encode the appearance and disappearance of topological features in the original filtration. Representation-finite quivers have been used in various applications of topological data analysis, such as \cite{Jia2022}.

In multiparameter persistence, introduced by \cite{FMhomotopy1999}, we consider filtrations coming from two or more parameters. The complexity jumps a significant amount. The work in \cite{CZmultidim2007} introduces the concept of multiparameter persistent homology and describes its algebraic structures as finitely-generated multi-graded modules over rings of multivariate polynomials. Various researchers have introduced frameworks to extend the concept of a persistence module to the multiparameter setting \cite{CZmultidim2007, BErealizations2018, BOOsigned2023} and others have contributed by proving that notions of distance and stability also nicely extend \cite{BLstability2023}. Other work in \cite{BErealizations2018} showed that multiparameter persistence modules arise from data via the Vietoris-Rips filtration.

One of the complications in multiparameter persistence from the view of quiver representation theory is the prevalence of quivers which are not representation-finite. In fact, even a relatively simple bi-filtration indexed over an $m$ by $n$ grid quickly becomes representation-infinite and soon after, even of wild representation type.

\subsection{Main Results}
In this paper, I will suggest a framework for studying subcategories $\mathcal{C}$ of $\rep_{k}(Q)$ for $Q$ representation-infinite for which it is computationally feasible to decompose any object in $\mathcal{C}$ into a direct sum of indecomposable objects from $\mathcal{C}$. 

To begin the summary of the main results of this paper, we first recall a theorem of Bongartz \cite{Bongartz}.

\begin{theorem} Let $R$ be a commutative ring and $\mathcal{A}$ an abelian $R$-linear category such that each morphism set in $\mathcal{A}$ has finite length as an $R$-module. Let $\mathcal{C}$ be a full subcategory of $\mathcal{A}$ closed under direct sums and kernels. Then two objects $M$ and $N$ of $\mathcal{C}$ are isomorphic if and only if the lengths of $\Hom(M,X)$ and $\Hom(N,X)$ coincide for all $X$ in $\mathcal{C}$. \end{theorem}

We will adjust Bongartz's theorem for the case where $R=k$ is an algebraically closed field, $A$ is the category of representations of a quiver $Q$, and $M, N$ are finite-dimensional representations of $Q$. Note that the category of representations of a quiver $Q$ is equivalent to the category of modules over the path algebra $k Q$ of that quiver. For a description of the functors defining this equivalence of categories and a proof, the reader may refer to \cite{Schiffler:2014aa}.

The proof analysis in my work refines Bongartz's theorem in two ways: by having fewer assumptions on $\mathcal{C}$, and giving a bound on $\dim X$ as demonstrated in the following theorem. The definition of the evaluation morphism $\ev[M][N]$ is given in \ref{def:evalmorph}.

\begin{corollary}[Corollary \ref{cor:specialC}]
    Let $k$ be a field, $A$ a finite-dimensional $k$-algebra, and let $\mathcal{C}$ be a subcategory of a $\rep_{k} (A)$ which is closed under direct sums and contains $\ker(\ev[M][N])$ for all $M,N$ in $\mathcal{C}$. Then, $M,N \in \mathcal{C}$ are isomorphic if and only if 
	\[\dim_{k} \Hom_{A}(M,X) = \dim_{k} \Hom_{A}(N,X),\]
    
    for all $X \in \mathcal{C}$ with $\dim X \leq (\dim M)^2 \dim N$.
\end{corollary}

The next result and its corollary were instrumental in using the evaluation morphism in conjunction with Auslander-Reiten theory for quivers. The first provides a description of the $\tau$-translate of the evaluation morphism. The corollary proves that the order of applying the Auslander-Reiten translation and taking the kernel of the evaluation morphism may be exchanged in certain situations.

\begin{theorem}{\ref{thm:tev_description}}
    For a quiver $Q$ and two representations $M,N \in \rep(Q)$, 
        \begin{equation}(\mathbbm{1}_{\tau M} \otimes \, \varphi) \circ \ev[\tau M][\tau N] = \tau \ev[M][N].\end{equation}
    where $\varphi:\Hom_{A}(M,N) \to \Hom_{A}(\tau M, \tau N)$ is defined on $f:M\to N$ by \[\varphi f = \bigoplus_{a\in Q_1} \mathbbm{1}_{I(ta)} \otimes \, f_{sa} \Big|_{\tau M}\, ,\]
    which is the restriction of the map
    \[\bigoplus_{a\in Q_1} \mathbbm{1}_{I(ta)} \otimes \, f_{sa}: \bigoplus_{a\in Q_1} I(ta) \otimes M_{sa} \to \bigoplus_{a\in Q_1} I(ta) \otimes N_{sa}
    \] to $\tau M$.
    In other words, the following diagram commutes:
        \begin{equation}\begin{tikzcd}[column sep=small]
        	{\tau M \otimes\Hom_{A}(M,N)} &&& {\tau N} \\
        	\\
        	\\
        	{\tau M \otimes\Hom_{A}(\tau M, \tau N)}
        	\arrow["{\tau \ev[M][N]}", from=1-1, to=1-4]
        	\arrow["{\ev[\tau M][\tau N]}"', from=4-1, to=1-4]
        	\arrow["{\mathbbm{1}_{\tau M}\otimes \, \varphi}"', dashed, from=1-1, to=4-1]
        \end{tikzcd}\end{equation}
\end{theorem}

\begin{customcor}{\ref{cor:ker_ev_t}}
    For two indecomposable representations $M, N \in \rep_{k}(Q)$ of a quiver $Q$ and $k \in \Z_{\geq 0}$ such that $\tau^{k-1} M, \tau^{k-1} N$ are not projective, we have \begin{equation}\ker \ev[M][N] \cong \tau^{-k} \ker \ev[\tau^k M][\tau^k N].\end{equation} Similarly, the statement holds for $k < 0$ if $\tau^{k+1} M, \tau^{k+1} N$ are not injective.
\end{customcor}

We aim to provide a framework wherein this adjusted theorem of Bongartz is a tool to detect isomorphism via the evaluation morphism. The evaluation morphism is compatible with the Auslander-Reiten translation and supports explorations for various quivers, such as the Kronecker quiver. 

\section{Background and Notation}
    We will generally follow notation as given in \cite{Schiffler:2014aa, DWbook} with a few changes in notation. In all situations, $k$ is an algebraically closed field.

\subsection{Quiver Representation Theory}

    A \textit{quiver} $Q = (Q_0,Q_1,s,t)$ consists of $Q_0$, a set of vertices, $Q_1$, a set of arrows, $s:Q_1 \to Q_0$, which maps an arrow to its starting vertex, and $t:Q_1 \to Q_0$, which maps an arrow to its terminal vertex. We will sometimes represent an element $\alpha \in Q_1$ by drawing an arrow from its starting vertex $s(\alpha)$ to its terminal vertex $t(\alpha)$ as $s(\alpha)\xto{\alpha} t(\alpha)$. A \textit{representation} $M=(M_x,\varphi_{\alpha})_{x\in Q_0, \alpha \in Q_1}$ of a quiver $Q$ is a collection of $k$-vector spaces $M_x$, one for each vertex $x\in Q_0$, and a collection of $k$-linear maps $\varphi_{\alpha}:M_{s(\alpha)}\to M_{t(\alpha)}$, one for each arrow $\alpha \in Q_1$. For a finite quiver $Q$, a representation $M$ is called \textit{finite-dimensional} if each vector space $M_x, x \in Q_0$, is finite-dimensional. The representation $M$ can then be described in terms of its \textit{dimension vector} $\ddim M = (\dim M_x)_{x \in Q_0}$.

    We will write concatenations of arrows in the same order as the composition of functions. So, for a path $p=\alpha_{\ell}\alpha_{\ell-1}\cdots\alpha_1$ comprised of arrows, we first apply $\alpha_1$, then $\alpha_2$, and so on. The set of all paths of a quiver generates over $k$ the \textit{path algebra }of $Q$. Given a quiver $Q$, we will use $\rep_{k} Q$ to denote the category of finite-dimensional representations of $Q$ along with the morphisms between them. Note that $\rep_{k} Q$ is an abelian $k$-category. The reader may refer to \cite{Schiffler:2014aa} for definitions of \textit{morphisms, direct sums, indecomposability} in the quiver setting.
    
    This paper is mainly concerned with representation-tame quivers. A quiver $Q$ is \textit{tame} if, for any dimension vector, all but finitely many isomorphism classes of indecomposable representations of that dimension can be described by a finite number of one-parameter families. An example of a tame quiver which we work with extensively in this paper is the Kronecker quiver.

\begin{definition} The Kronecker quiver $\mathsf{K}$ is a quiver with two vertices $x,y$ and two arrows $\alpha, \beta$ oriented in the same direction:

\[\begin{tikzcd}
	x & y
	\arrow["\alpha", shift left=1, from=1-1, to=1-2]
	\arrow["\beta"', shift right=1, from=1-1, to=1-2]
\end{tikzcd}\]

 The quiver $\mathsf{K}$ is of tame representation type. Its indecomposable representations are divided into three components: the postprojective indecomposables  $\mathcal{P} = \{P_n, \, n\in \Z_{\geq 0}\}$, the preinjective indecomposables $\mathcal{I}=\{I_n, \, n\in \Z_{\geq 0}\}$, and the regular indecomposables $\mathcal{R}=\{R_n(\lambda), \, n\in \Z_{\geq 0}, \, \lambda \in k \cup \{\infty\}\}$.

Let $\mathbbm{1}_n$ denote the $n \times n$ identity matrix. For the postprojective indecomposable representations $P_n: k^n \rightrightarrows k^{n+1} \in\mathcal{P}$, we have $\ddim P_n = (n, n+1)$ and maps $(\alpha_{P_n},\beta_{P_n})$ where 
    \[  \alpha_{P_n} = \begin{bmatrix} \mathbbm{1}_n\\ 0 \end{bmatrix}, \quad 
        \beta_{P_n} = \begin{bmatrix} 0\\ \mathbbm{1}_n \end{bmatrix} 
    \]
For the preinjective indecomposable representations $I_n: k^{n+1} \rightrightarrows k^n \in\mathcal{I}$, we have $\ddim I_n = (n+1, n)$ and maps $(\alpha_{I_n},\beta_{I_n})$ where
    \[  \alpha_{I_n} = \begin{bmatrix} \mathbbm{1}_n & 0 \end{bmatrix}, \quad 
        \beta_{I_n} = \begin{bmatrix} 0 & \mathbbm{1}_n \end{bmatrix} 
    \]
        
For the regular indecomposable representations $R_n(\lambda): k^n \rightrightarrows k^n \in\mathcal{R}$, we have $\ddim R_n(\lambda) = (n, n)$ and maps $(\alpha_{R_n(\lambda)},\beta_{R_n(\lambda)})$ where 
    \[  \alpha_{R_n(\lambda)} = \begin{cases} \mathbbm{1}_n, & \lambda \neq \infty\\ J_n(0), & \lambda = \infty \end{cases}, \qquad 
        \beta_{R_n(\lambda)} = \begin{cases} J_n(\lambda), & \lambda \neq \infty\\ \mathbbm{1}_n, & \lambda = \infty \end{cases},
    \]
    with $J_n(\lambda)$ the Jordan block of size $n$ with eigenvalue $\lambda$. 
\end{definition}
For conciseness and whenever there is no confusion, we refer to the maps $\alpha_M, \beta_M$ for a representation $M$ as $\alpha, \beta$.

\subsection{Homology of Quiver Representations}\label{sec:homology}
This section reviews some definitions of homology of quiver representation that will be useful for the results in Chapters 4 and 5.

\begin{definition} 
Let $x$ be a vertex of the quiver $Q$. Define representations $P(x)$ and $I(x)$ as follows:
\begin{enumerate}[(a)]
    \item $P(x) = (P(x)_y, \varphi_{\alpha})_{y\in Q_0, \alpha\in Q_1}$ 
    where $P(x)_y$ is the $k$-vector space with basis the set of all paths out of $x$ to $y$ in $Q$; so the elements of $P(x)_y$ are of the form $\sum_c \lambda_c c$, where $c$ runs over all paths from $x$ to $y$, and $\lambda_c \in k$; and if $y \xrightarrow{\alpha} z$ is an arrow in $Q$, then $\varphi_{\alpha}:P(x)_y \to P(x)_z$ is the linear map defined on the basis by composing the paths from $x$ to $y$ with the arrow $y \xrightarrow{\alpha} z$.
    
    $P(x)$ is called the \textbf{projective representation} at vertex $x$.

    \item $I(x) = (I(x)_y, \varphi_{\alpha})_{y\in Q_0, \alpha\in Q_1}$ 
    where $I(x)_y$ is the $k$-vector space with basis the set of all paths from $y$ into $x$ in $Q$; so the elements of $I(x)_y$ are of the form $\sum_c \lambda_c c$, where $c$ runs over all paths from $y$ to $x$, and $\lambda_c \in k$;
    and if $y \xrightarrow{\alpha} z$ is an arrow in $Q$, then $\varphi_{\alpha}:I(x)_y \to I(x)_z$ is the linear map defined on the basis by deleting the arrow $y \xrightarrow{\alpha} z$ from those path from $y$ to $x$ which start with $\alpha$ and sending to zero the paths that do not start with $\alpha$.
    
    $I(x)$ is called the \textbf{injective representation} at vertex $x$.
\end{enumerate}
\end{definition}



 The objects $P(x)$ and the objects $I(x)$ for $x \in Q_0$ are projective and injective, respectively. Both $P(x)$ and $I(x)$ are indecomposable representations as proven in \cite{Schiffler:2014aa}. The following proposition provides a description of some particular $\Hom$-spaces involving the representations $P(x)$ and $I(x)$.

\begin{proposition}\label{prop:hom_isos}
    For two representations $M, N$ of a quiver $Q$ and $P(x), I(x)$ as defined above for vertices $x \in Q_0$, we have 
        \begin{enumerate}[(a)]
            \item $\Hom_Q(P(x),N) \cong N(x)$
            \item $\Hom_Q(M,I(x)) \cong M(x)^*$
        \end{enumerate}
\end{proposition}




 We continue to define the standard projective resolution which will feature prominently in the proofs of the main results in Chapter 5.

\begin{proposition}\label{prop:stdres} For any left $k Q$-module $M$, we have an exact sequence of $k Q$-modules 
    
    \begin{equation}\label{eqn:stdprojres}
        0 \rightarrow \bigoplus_{\alpha \in Q_1} P(t(\alpha)) \otimes_{k} e_{s(\alpha)}M \xrightarrow{u} \bigoplus_{x\in Q_0} P(x) \otimes_{k} e_x M \xrightarrow{v} M \rightarrow 0,
    \end{equation}
    
    where the maps $u$ and $v$ are defined by 
        \[u(a \otimes m) \coloneqq a\alpha \otimes m - a \otimes \alpha m, \quad a\in P(t(\alpha)), m \in e_{s(\alpha)}M\] and 
        \[v(a \otimes m) \coloneqq am, \quad a\in P(x), m\in e_x M.\] 
        
    Here each $P(y) \otimes_{k} e_x M$ is a $k Q$-module via $a(b\otimes m)=ab \otimes m$, where $a \in k Q, b \in P(y),$ and $m \in e_x M$.
\end{proposition}
We call the sequence \ref{eqn:stdprojres} the \textbf{standard projective resolution} of the representation $M$.

\subsection{Auslander-Reiten Theory}\label{sec:ARtheory}

Let $\nu: \rep_{k} Q \to \rep_{k} Q$ be the \textbf{Nakayama functor}. The restriction of $\nu$ to $\text{proj}(Q)$ is an equivalence of the categories of projective representations of $Q$, $\text{proj}(Q)$, and  injective representations, $\text{inj}(Q)$, whose quasi-inverse is given by \[\nu^{-1}=\Hom(DA^{\text{op}},-):\text{inj}(Q)\to\text{proj}(Q).\] 
Moreover, for any vertex $x$, \[\nu P(x) = I(x)\] 
and if $c$ is a path from $x$ to $y$, and $f_c \in \Hom_Q(P(y),P(x))$ is the corresponding morphism, then \[\nu f_c : I(y) \to I(x)\] is the morphism given by the cancellation of the path $c$. 
 
\begin{definition}\label{def:tau}
Let \[0\to P_1 \xrightarrow{p_1} P_0 \xrightarrow{p_0} M \to 0\] be a minimal projective resolution. Applying the Nakayama functor, we get an exact sequence 
\[0\to \tau M \to \nu P_1 \xrightarrow{\nu p_1} \nu P_0 \xrightarrow{\nu p_0}  \nu M \to 0\]
where $\tau M = \ker(\nu p_1)$ is called the Auslander-Reiten translate of $M$ and $\tau$ is called the \textbf{Auslander-Reiten translation}.
\end{definition}

\begin{definition}
Let \[0\to M \xrightarrow{i_0} I_0 \xrightarrow{i_1} I_1 \to 0\] be a minimal injective resolution. Applying the inverse Nakayama functor, we get an exact sequence 
\[0\to \nu^{-1} M \xrightarrow{\nu^{-1} i_0} \nu^{-1} I_0 \xrightarrow{\nu^{-1} i_1} \nu^{-1} I_1 \to  \tau^{-1} M \to 0\]
where $\tau^{-1} M = \coker(\nu^{-1} i_1)$ is called the inverse Auslander-Reiten translate of $M$ and $\tau^{-1}$ the \textbf{inverse Auslander-Reiten translation}.
\end{definition}

The reader may refer to the sources \cite{Schiffler:2014aa, DWbook} more information on Auslander-Reiten theory, as well as a definition of the Auslander-Reiten quiver (AR-quiver).

\begin{example}
 The three components of the Kronecker quiver $\mathcal{P}$, $\mathcal{I}$, $\mathcal{R}$ are organized in the Auslander-Reiten quiver of $\mathsf{K}$ as show in figure~\ref{fig:AR(K)} below. The action of the Auslander-Reiten translation $\tau$ on the indecomposables is also represented.

\begin{figure}[ht]
\caption{$\AR(\mathsf{K})$ for the Kronecker quiver}
\[\begin{tikzcd}[ampersand replacement=\&,column sep=tiny]
	\&\&\&\&\& {} \\
	\&\&\&\&\& {R_3(\lambda)} \arrow["{\tau }"' swap,loop right, dashed, color={rgb,255:red,214;green,153;blue,92},looseness=3]\\
	\& {P_1} \& {} \& {P_3} \& {} \& {R_2(\lambda)} \arrow["{\tau }"' swap,loop right, dashed, color={rgb,255:red,214;green,153;blue,92},looseness=3] \&\&\& {I_2} \&\& {I_0} \\
	{P_0} \&\& {P_2} \&\&\& {\underbrace{\phantom{++;} R_1(\lambda) \phantom{++;}}_{\lambda \in k \cup \{\infty\}}} \arrow["{\tau }"' swap,loop right, dashed, color={rgb,255:red,214;green,153;blue,92},looseness=3,start anchor={[xshift=-4ex,yshift=2.5mm]east}, end anchor={[xshift=-4ex, yshift=-1mm]east}]\& {} \& {I_3} \&\& {I_1} \\
	\&\&\&\&\& {} 
	\arrow[shift left=1, from=4-1, to=3-2]
	\arrow[shift left=1, from=3-2, to=4-3]
	\arrow[shift left=1, from=4-3, to=3-4]
	\arrow[shift left=1, from=4-6, to=3-6]
	\arrow[shift left=1, from=4-8, to=3-9]
	\arrow[shift left=1, from=3-9, to=4-10]
	\arrow[shift left=1, from=4-10, to=3-11]
	\arrow[shift right=1, from=4-1, to=3-2]
	\arrow[shift right=1, from=3-2, to=4-3]
	\arrow[shift right=1, from=4-3, to=3-4]
	\arrow[shift right=1, from=4-8, to=3-9]
	\arrow[shift right=1, from=3-9, to=4-10]
	\arrow[shift right=1, from=4-10, to=3-11]
	\arrow[shift left=1, from=3-6, to=2-6]
	\arrow[shift left=1, from=2-6, to=3-6]
	\arrow[shift left=1, from=3-6, to=4-6]
	\arrow["\tau"', color={rgb,255:red,214;green,153;blue,92}, dashed, from=3-4, to=3-2]
	\arrow["\tau"', color={rgb,255:red,214;green,153;blue,92}, dashed, from=4-3, to=4-1]
	\arrow["{\tau }"', color={rgb,255:red,214;green,153;blue,92}, dashed, from=3-11, to=3-9]
	\arrow["\tau"', color={rgb,255:red,214;green,153;blue,92}, dashed, from=4-10, to=4-8]
	\arrow["\cdot \,\, \cdot \,\, \cdot"{description}, shift right=5, draw=none, from=4-8, to=4-7]
	\arrow["\cdot \,\, \cdot \,\, \cdot"{description}, shift right=5, draw=none, from=3-4, to=3-5]
	\arrow["\vdots"{description}, draw=none, from=2-6, to=1-6]
\end{tikzcd}\]
\label{fig:AR(K)}
\end{figure}
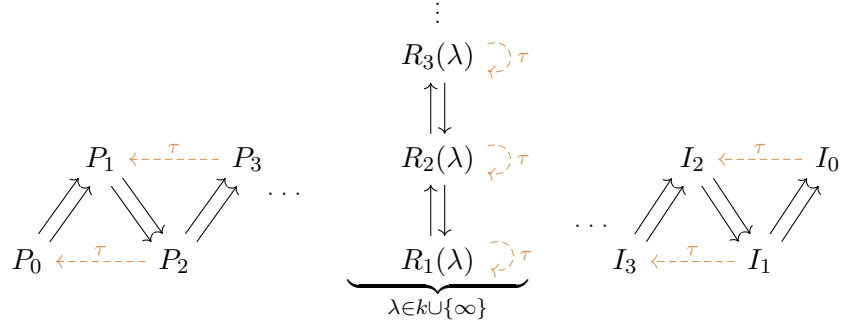
\end{example}

\begin{corollary}\label{cor:tau_isos} 
    For two indecomposable representations $M, N \in \rep \mathsf{K}$ and an integer $k\geq 0$, then
        \begin{enumerate}[(a)]
            \item $\Hom_Q(M,N) \cong \Hom_Q(\tau^k M, \tau^k N)$ if neither $M$ nor $N$ is projective.
            \item $\Hom_Q(M,N) \cong \Hom_Q(\tau^{-k} M, \tau^{-k} N)$ if neither $M$ nor $N$ is injective.
        \end{enumerate}
\end{corollary}

These results follow from the following statements given in \cite{DWbook} Proposition 6.4.3.
    \begin{enumerate}[(i)]
        \item If $M$ and $N$ both have no projective summands then \[\Hom_Q(M,N) \cong \Hom_Q(\tau M, \tau N)\] 
        \item If $M$ and $N$ both have no injective summands, then \[\Hom_Q(M,N) \cong \Hom_Q(\tau^{-1} M,\tau^{-1} N)\]
    \end{enumerate}

\subsection{Derived Categories}\label{sec:derivedcateg}
Next, we collect some definitions and results from \cite{Happel} about triangulated and derived categories.

Let $\mathcal{A}$ be an additive category with an automorphism $T$ with inverse $T^{-1}$. $T$ is called the \textbf{translation functor} and translates an object $X$ to $TX$, sometimes written $X[1]$. The inverse acts by translating $X$ to $X[-1]$. A \textbf{triangulation} of $\mathcal{A}$ is a collection of \textbf{triangles} $X\xto{u} Y \xto{v} Z \xto{w} X[1]$ such that the conditions (TR1)-(TR4) of \cite{Happel} are satisfied. The category $\mathcal{A}$ is then called a \textbf{triangulated category}. For convenience, we list (TR1) and (TR2) below.

\begin{definition}\label{def:triangles}  
Let $X,Y,Z$ be objects in $\mathcal{A}$ and $u:X\to Y, v:Y \to Z, w:Z\to X[1]$ be morphisms between them. 

    \begin{itemize}      
        \item[(TR1)] Every triangle isomorphic to a triangle is a triangle. 
        Every morphism $u: X \rightarrow Y$ in $\mathcal{A}$, can be embedded into a triangle $X \xrightarrow{u} Y \xrightarrow{v} Z \xrightarrow{w} X[1]$. 
        The sextuple $X \xrightarrow{\mathbbm{1}_X} X \xrightarrow{0} 0 \xrightarrow{0} X[1]$ is a triangle. 
        
        \item[(TR2)] If $X \xrightarrow{u} Y \xrightarrow{v} Z \xrightarrow{w} X[1]$ is a triangle, then $Y \xrightarrow{v} Z \xrightarrow{w} X[1] \xrightarrow{-u[1]} Y[1]$ is a triangle.  
    \end{itemize}
\end{definition}

 We remark that \cite{Happel} also tells us that the converse of (TR2) is true. 

\begin{definition}\label{def:complex}
    Let $\mathcal{C}$ be an abelian category.  A \textbf{(differential) complex} over $\mathcal{C}$ is $X^{\bullet} = (X^i,d_X^i)_{i\in \Z}$, a collection of objects $X^i$ of $\mathcal{C}$ and morphisms $d_X^i:X^i\to X^{i+1}$ (called differentials) such that $d_X^i d_X^{i+1} = 0$. The complex is called a \textbf{stalk complex} if there exists an $i_0$ such that $X^{i_0} \neq 0$ and $X^i = 0$ for all $i \neq i_0$. Then $X^{i_0}$ is called a \textbf{stalk} and we say that the complex is concentrated in degree $i_0$.

    The complex is \textbf{bounded} if $X^i = 0$ for all but finitely many $i < 0$ and for all but finitely many $i > 0$.
\end{definition}

\begin{definition}
    For two complexes $X^{\bullet}=(X^i,d_X^i), Y^{\bullet}=(Y^i,d_Y^i)$, a morphism between them is $f^{\bullet}$, given by a sequence of morphisms $f^i:X^i\to Y^i$ such that $d_X^i f^{i+1} = f^i d_Y^i$ for all $i \in \Z$. 
\end{definition}

\begin{definition}\label{def:cone}
    Define the \textbf{(mapping) cone} $C_f^{\bullet}$ of a morphism $f^{\bullet}:X^{\bullet} \to Y^{\bullet}$ is the complex \[C_f^{\bullet} = ((X[1]^{\bullet})^i \oplus Y^i, d_{C_f}^i)\] with differential
        \[X^{i+1}\oplus Y^i \xto{\begin{bmatrix}-d_X^{i+1} & f^{i+1} \\ 0 & d_Y^i \end{bmatrix}} X^{i+1}\oplus Y^{i+1}\]
\end{definition}

The mapping cone embeds the morphism $f^{\bullet}$ into a triangle 
    \[X^{\bullet} \xto{f^{\bullet}} Y^{\bullet} \to C_f^{\bullet} \to X[1]^{\bullet}.\] 
Applying the converse of (TR2) to this triangle, we get another triangle 
    \[C_f[-1]^{\bullet} \to X^{\bullet} \xto{f^{\bullet}} Y^{\bullet} \to C_f^{\bullet} .\] 
We refer to $C_f[-1]^{\bullet}$ as the \textbf{co-cone} of the morphism $f$ and, following \cite{Sanders}, we use the co-cone as a ``kernel" of the morphism $f$. In fact, for any triangle $$X^{\bullet} \xto{f^{\bullet}} Y^{\bullet} \to Z^{\bullet} \to X[1]^{\bullet},$$ we will use $Z[-1]^{\bullet}$ as the co-cone.

If $X, Y$ are stalk complexes (i.e. $X,Y \in \mathcal{C}$), we note that the mapping cone $C_f$ of a morphism $f:X\to Y$ is only nonzero in degrees $-1$ and $0$ with $C_f^{-1} = X, C_f^0=Y$ and $d_{C_f}^{-1} = \begin{bmatrix}0 & f\\ 0 & 0\end{bmatrix} \cong f$.

\begin{definition}\label{def:derivedcateg}
    For $\mathcal{C}$ an abelian category, the \textbf{bounded derived category} $D^b(\mathcal{C})$ is the category of bounded complexes over $\mathcal{C}$ with quasi-isomorphism inverted (see \cite{Happel} for details).
\end{definition}

It turns out that $D^b(\mathcal{C})$ is a triangulated category
and we can embed $\mathcal{C}$ in $D^b(\mathcal{C})$ by sending each object $X \in \mathcal{C}$ to a stalk complex concentrated at $i_0 = 0$, with $X^0 = X$ and $X^i=0$ for all other $i$.
Whenever there is no confusion, we will write $X$ to mean the stalk complex $X^{\bullet}$ concentrated in degree $0$.

When $\mathcal{C}=\rep_k (A)$ for a finite-dimensonal algebra $A$, \cite{Happel} tells us that we have a special triangle, referred to as an Auslander-Reiten triangle given by 
    \[\nu X[-1] \to Y \to X \to \nu X\]
for $X$ indecomposable. Auslander-Reiten triangles are related to Auslander-Reiten sequences, almost-split exact sequences of the form 
    \[0 \to X \to Y \to Z \to 0\]
where $X, Z$ are indecomposable and $X \cong \tau Z$. The reader may refer to the source for a more in-depth treatment. Most importantly, the triangle above being an Auslander-Reiten triangle tells us that \[\tau X \cong \nu X[-1]\] in the derived category. Therefore, the Auslander-Reiten translation is a composition of two auto-equivalences $\nu, T$ on the derived category and is thus an autoequivalence itself. Therefore, we get the following, stated in the next lemma for convenience:

\begin{lemma}\label{lem:tau_autoequiv}
    The Auslander-Reiten translation $\tau$ is an autoequivalence on the derived category $D^b(\mathcal{C})$ and preserves triangles.
\end{lemma}

 In general, we may associate the short exact sequence \[0 \to X \xto{u} Y \xto{v}  Z \to 0\] with the triangle \[X \xto{u} Y \xto{v} Z \xto{w} X[1]\] for some object $Z$ and some morphism $w:Z \to X[1]$.

\section{Proof of the Main Theorem}

To begin, let us recall two theorems by Auslander and Bongartz \cite{Bongartz} which provide criteria for distinguishing isomorphism classes of certain modules over a ring. 

\begin{theorem} Let $R$ be a commutative ring and $\mathcal{A}$ an abelian $R$-linear category such that each morphism set in $\mathcal{A}$ has finite length as an $R$-module. Let $\mathcal{C}$ be a full subcategory of $\mathcal{A}$ closed under direct sums and kernels. Then two objects $M$ and $N$ of $\mathcal{C}$ are isomorphic if and only if the lengths of $\Hom(M,X)$ and $\Hom(N,X)$ as $R$-modules coincide for all $X$ in $\mathcal{C}$. 
\end{theorem}

\begin{proof}
    Let $M$ and $n$ be given in $\mathcal{C}$ such that for all $X$ in $\mathcal{C}$ the lengths of $\Hom(M,X)$ and $\Hom(N,X)$ are equal. First of all, we claim that $M$ and $N$ have a nonzero direct summand in common. To provide this we take generators $f_1,f_2,\ldots,f_n$ of $\Hom(M,N)$ as an $R$-module and we look at the map $f:M^n \to N$ given by 
        \[f(m_1,m_2,\ldots,m_n)=f_1(m_1) + f_2(m_2) + \ldots + f_n(m_n).\]
    Clearly, the kernel $K$ of $f$ belongs to $\mathcal{C}$. By construction, the exact sequence 
        \[0 \to K \to M^n \to N\]
    induces an exact sequence
        \[0 \to \Hom(M,K) \to \Hom(M,M^n),\to \Hom(M,N) \to 0.\]
    Counting lengths, we see that the induced sequence
        \[0 \to \Hom(N,K) \to \Hom(N,M^n),\to \Hom(N,N) \to 0.\]
    has to be exact too. Therefore, $f$ is a retraction, when $N$ is a direct summand of $M^n$. It follows easily from our assumptions on $\mathcal{A}$ and $\mathcal{C}$, that any object in both categories is a finite direct sum of indecomposable objects having local endomorphism rings. Hence the Krull-Remak-Schmidt theorem holds in $\mathcal{A}$ and $\mathcal{C}$. In particular, $N$ and $M$ have a common indecomposable direct summand because $N$ is a direct summand of $M^n$.

    The proof of the theorem proceeds by induction on the length of $\Hom(N,N)$, the case $\Hom(N,N)=0$ being trivial. By the argument given before, we may write $M=U\oplus M'$ and $N=U\oplus N'$ with some objects $U,M'$ and $N'$ in $\mathcal{C}$ satisfying $U\neq 0$. Of course, the lengths of $\Hom(M',X)$ and $\Hom(N',X)$ still coincide for all $X$ in $\mathcal{C}$. Since the length of $\Hom(N',N')$ is strictly smaller than the length of $\Hom(N,N)$, we get by induction that $M'$ and $N'$ are isomorphic, whence the same is true for $M$ and $N$.
\end{proof}

We will extract the map $f$ defined in the proof above and adapt it to the situation of quiver representations. 

\begin{definition}\label{def:evalmorph}
    We define the \textbf{evaluation morphism} $\ev[M][N]:M^n \rightarrow N$ for two representations $M$ and $N$ in $\rep_{k} Q$ for a quiver $Q$ and a field $k$ as the following map
    
    \[\ev[M][N](m_1,m_2,\ldots,m_n)=f_1(m_1) + f_2(m_2) + \ldots + f_n(m_n).\]
\end{definition}

 In the setting of quiver representations, we replace the commutative ring $R$ with the field $k$ and take $\mathcal{A}$ to be $\rep_{k} Q$ which is an abelian $k$-linear category. When the basis of $\Hom_Q(M,N)$ is not important, we write 
    \[\ev[M][N]: M \otimes_{k} \Hom_{Q}(M,N) \to N.\]
The following refines Bongartz's in two ways: by noticing that Bongartz's proof only requires that the kernels of $\ev[M][N]$ are in $\mathcal{C}$ and by providing a bound on $\dim X$.

\begin{corollary}\label{cor:specialC}
    Let $k$ be a field and let $\mathcal{C}$ be a full subcategory of a $\rep_{k} Q$ which is closed under direct sums and contains $\ker(\ev[M][N])$ for all $M,N$ in $\mathcal{C}$. Then, $M,N \in \mathcal{C}$ are isomorphic if and only if 
	\[\dim_{k} \Hom_{Q}(M,X) = \dim_{k} \Hom_{Q}(N,X),\]
    for all $X \in \mathcal{C}$ with $\dim X \leq (\dim M)^2 \dim N$.
\end{corollary}

Categories $\mathcal{C}$ which satisfy the hypotheses of Corollary~\ref{cor:specialC}, will be referred to as \textbf{completely-decomposable} subcategories of $\rep_{k} Q$. 

For the first part of the corollary, we note that following Bongartz's proof only requires $\ker \ev[M][N]$ rather than $\ker f: M\to N$ for all $M, N$ in $\mathcal{C}$. In the setting of quiver representations, requiring only the kernels of the evaluation morphism rather than all kernels does indeed reduce the number of objects required to be in $\mathcal{C}$. 

\begin{example}
For example, consider the Kronecker quiver $Q=\mathsf{K}$ and let $\mathcal{C}$ be the category of representations which are direct sums of preinjective indecomposable representation $I_m$, $m \geq 0$. Let us calculate the kernel of the map $f:I_1 \to I_0$ where $f$ is given by the pair of maps $(\begin{bsmallmatrix} \lambda & -1\end{bsmallmatrix}, 0)$ for some $\lambda \in k$.     

The kernel of this map is the representation $R_1(\lambda)$, because one can see that the dimension vector of the kernel must be $(1,1)$ with maps $\lambda$ and $1$ over the arrows. Thus, requiring this kernel in $\mathcal{C}$ would require $R_1(\lambda)$.  On the other hand, the kernel of the evaluation morphism between $I_1$ and $I_0$ is $I_2$, which is already in the category we described initially. As we will see in the next section, it is true for the Kronecker quiver that the kernels of the evaluation morphism between indecomposables in the same component of the $AR$-quiver stay within that component.
\end{example}

Secondly, we highlight the fact that we only need to calculate $\dim_{k} \Hom_{Q}(M,X)$ for $X = \ker \ev[M][N]$, $M$, and $N$ in Bongartz's proof. Since $\ker \ev[M][N] \subseteq M \otimes \Hom(M,N)$ and $\Hom(M,N)$ has dimension bounded by $(\dim M)(\dim N)$, we can bound the dimension of the kernel by $(\dim M)^2\dim N$. It is clear that $M$ and $N$ also satisfy this bound.

A natural problem is then to find subcategories $\mathcal{C}$ which satisfy the hypotheses of the theorem. I am particularly interested in the case that $\mathcal{A}$ is a quotient of a quiver path algebra, and $\mathcal{C}$ is a subcategory of modules which naturally arise in persistence theory.
Since little is known about properties of $\ev[M][N]$, there is much work to be done in identifying such $\mathcal{C}$. We will provide examples of such a subcategory for the Kronecker quiver.

\subsection{The evaluation morphism and the Auslander-Reiten \mbox{translation}}
The main theorem of this section describes how the Auslander-Reiten translation $\tau$ (Definition~\ref{def:tau}) interacts with evaluation morphisms. It was first discovered through observations with the Kronecker quiver test case which will prove significantly useful in generalizing Corollary~\ref{cor:specialC} to other quivers. 

\begin{proposition}\label{prop:tau_f}
    Given two representations $M,N$ of a quiver $Q$ and morphism $f:M \to N$, we have
    \begin{enumerate}[(a)]
    \item The AR-translate $\tau M$ is a subset of the space $\bigoplus\limits_{a\in Q_1} I(ta)\otimes M_{sa}$
    \item The AR-translate of $f$ can be written as \begin{equation}\tau f = \bigoplus_{a\in Q_1} \mathbbm{1}_{I(ta)} \otimes \, f_{sa} \Big|_{\tau M}\, ,\end{equation}
        which is the restriction of the map
            \begin{equation}\bigoplus_{a\in Q_1} \mathbbm{1}_{I(ta)} \otimes \, f_{sa}: \bigoplus_{a\in Q_1} I(ta) \otimes M(sa) \to \bigoplus_{a\in Q_1} I(ta) \otimes N(sa)\end{equation} to $\tau M$.
    \end{enumerate}
\end{proposition}

\begin{proof}
We begin by taking the standard projective resolutions for $M, N$ as in Proposition~\ref{prop:stdres}, with 
    \[u_M(a \otimes m) \coloneqq a\alpha \otimes m - a \otimes \alpha m, \quad a\in P(t(\alpha)), m \in e_{s(\alpha)}M\] and \[v_M(a \otimes m) \coloneqq am, \quad a\in P(x), m\in e_x M,\] and similarly for $N$.
We connect the two resolutions by $f$ and lifts 
    \[f' = \bigoplus_{x\in Q_0} \mathbbm{1}_{P(x)} \otimes f_x \quad \text{and} \quad f'' = \bigoplus_{a \in Q_1} \mathbbm{1}_{P(ta)} \otimes f_{sa}.\]

The lifts $f'$ and $f''$ are defined by the commutativity of the following diagram:
\begin{equation}\begin{tikzcd}[ampersand replacement=\&]
	M \&\& {\bigoplus_{x\in Q_0} \, P(x) \otimes M_x} \&\& {\bigoplus_{a\in Q_1} \, P(ta) \otimes M_{sa}} \\
	\\
	\\
	N \&\& {\bigoplus_{x\in Q_0} \, P(x) \otimes N_x} \&\& {\bigoplus_{a\in Q_1} \, P(ta) \otimes N_{sa}}
	\arrow["{f \,}"', dashed, from=1-1, to=4-1]
	\arrow["{u_M}"', from=1-5, to=1-3]
	\arrow["{f''}"', dashed, from=1-5, to=4-5]
	\arrow["{f'}"', dashed, from=1-3, to=4-3]
	\arrow["{u_N}", from=4-5, to=4-3]
	\arrow["{v_N}", from=4-3, to=4-1]
	\arrow["{v_M}"', from=1-3, to=1-1]
\end{tikzcd}\end{equation}

 It follows that
    \begin{nalign}\label{eq:comm_rels}
        fv_M &= v_Nf' \\
        f'u_M &= u_N f''
    \end{nalign}

The original function $f=(f_x:M(x) \to N(x))_{x \in Q_0}$ contains data for maps between $M(x)$ and $N(x)$ at each vertex $x \in Q_0$. We may also use the identity map $\mathbbm{1}_{P(x)}$ in the first argument of each tensor. To show that these do indeed satisfy the commutativity relations \ref{eq:comm_rels}:
    \begin{align*}
        fv_M(p\otimes m) &= f(pm) \\
                         &= pf_x(m)\nonumber\\
                         & \nonumber\\
        v_Nf'(p\otimes m) &= v_N ( \oplus_{x\in Q_0} \mathbbm{1}_{P(x)} \otimes f_x)(p \otimes m)\\
        &= v_N(p\otimes f_x(m))\nonumber\\
        &= pf_x(m)\nonumber
    \end{align*}
for $p \in P(x), m\in M(x)$. Also,
    \begin{align*}
        f'u_M(p\otimes m) &= f'(pa \otimes m - a \otimes pm) \\
                          &= (\oplus_{a \in Q_1} \mathbbm{1}_{P(ta)} \otimes f_{sa})(pa \otimes m - a \otimes pm)\nonumber\\
                          &= pa \otimes f_{sa}(m) - a \otimes f_{sa}(pm)\nonumber\\
                          &= pa \otimes f_{sa}(m) - a \otimes pf_{sa}(m)\nonumber\\
                          & \nonumber\\
        u_Nf''(p\otimes m) &= u_N\oplus_{a \in Q_1} \mathbbm{1}_{P(ta)} \otimes f_{sa})(p \otimes m)\\
                           &= u_N(p\otimes f_{sa}(m))\nonumber\\
                           &= pa \otimes f_{sa}(m) - a \otimes pf_{sa}(m)\nonumber
    \end{align*}
for $p \in P(ta), m \in M(sa)$. Therefore, we have a morphism $(f,f',f'')$ between the standard resolutions of $M$ and $N$. 

\begin{equation}\begin{tikzcd}[ampersand replacement=\&]
	M \&\& {\bigoplus_{x\in Q_0} \, P(x) \otimes M_x} \&\& {\bigoplus_{a\in Q_1} \, P(ta) \otimes M_{sa}} \\
	\\
	\\
	N \&\& {\bigoplus_{x\in Q_0} \, P(x) \otimes N_x} \&\& {\bigoplus_{a\in Q_1} \, P(ta) \otimes N_{sa}}
	\arrow["{f \,}"', dashed, from=1-1, to=4-1]
	\arrow["{u_M}"', from=1-5, to=1-3]
	\arrow["{\bigoplus_{a\in Q_1} \, \mathbbm{1}_{P(ta)} \otimes f_{sa}}"', dashed, from=1-5, to=4-5]
	\arrow["{\bigoplus_{x\in Q_0} \, \mathbbm{1}_{P(x)}\otimes f_x}"', dashed, from=1-3, to=4-3]
	\arrow["{u_N}", from=4-5, to=4-3]
	\arrow["{v_N}", from=4-3, to=4-1]
	\arrow["{v_M}"', from=1-3, to=1-1]
\end{tikzcd}\end{equation}

Next, we apply the Nakayama functor to the diagram to obtain injective presentations of $\tau M, \tau N$:
\begin{equation}\begin{tikzcd}[ampersand replacement=\&]
	{\bigoplus_{x\in Q_0} \, I(x) \otimes M_x} \&\& {\bigoplus_{a\in Q_1} \, I(ta) \otimes M_{sa}} \&\& {\tau M} \\
	\\
	\\
	{\bigoplus_{x\in Q_0} \, I(x) \otimes N_x} \&\& {\bigoplus_{a\in Q_1} \, I(ta) \otimes N_{sa}} \&\& {\tau N}
	\arrow["{\tau f = \bigoplus_{a\in Q_1} \mathbbm{1}_{I(ta)} \otimes f_{sa} \Big|_{\tau M}}"', dashed, from=1-5, to=4-5]
	\arrow["{\nu u_M}"', from=1-3, to=1-1]
	\arrow["{\bigoplus_{a\in Q_1} \, \mathbbm{1}_{I(ta)} \otimes f_{sa}}"', dashed, from=1-3, to=4-3]
	\arrow["{\bigoplus_{x\in Q_0} \, \mathbbm{1}_{I(x)}\otimes f_x}"', dashed, from=1-1, to=4-1]
	\arrow[from=1-5, to=1-3]
	\arrow[from=4-5, to=4-3]
	\arrow["{\nu u_N}", from=4-3, to=4-1]
\end{tikzcd}\end{equation}

Because $\tau M, \tau N$ inject into the spaces to their respective left, we conclude that $\tau f$ is the restriction of $\nu f''$ which can be represented by the restriction 
\[\tau f = \bigoplus_{a \in Q_1} \mathbbm{1}_{I(ta)} \otimes f_{sa}\big|_{\tau M} \colon \tau M \to \tau N. \qedhere\]
\end{proof}

We can now present the main theorem of this section.

\begin{theorem}\label{thm:tev_description}
    For a quiver $Q$ and two representations $M,N \in \rep_{k} Q$ without any projective summands, 
        \[\ev[\tau M][\tau N] \circ (\mathbbm{1}_{\tau M} \otimes \, \varphi)= \tau \ev[M][N].\]
    
    where $\varphi$ is an isomorphism from $\Hom(M,N)$ to $\Hom(\tau M, \tau N)$ which exists by Lemma~\ref{cor:tau_isos}$(a)$. In other words, the following diagram commutes:
\begin{equation}\label{dgm:triangle}
    \begin{tikzcd}[ampersand replacement=\&]
	{\tau M \otimes \Hom(M,N)} \&\&\& {\tau N} \\
	\\
	\& {} \\
	{\tau M \otimes \Hom(\tau M,\tau N)}
	\arrow["{\tau \ev[M][N]}", from=1-1, to=1-4]
	\arrow["{\ev[\tau M][\tau N]}"', from=4-1, to=1-4]
	\arrow["{\mathbbm{1}_{\tau M}\otimes \varphi}"', dashed, from=1-1, to=4-1]
\end{tikzcd}\end{equation}
\end{theorem}

\begin{proof} 
By the previous proposition, we obtain a description of $\tau \ev[M][N]$:
    \begin{equation}\tau \ev[M][N] = \bigoplus_{a \in Q_1} \mathbbm{1}_{I(ta)} \otimes \ev_{sa} \big|_{\tau M \otimes \Hom(M,N)} : \tau M \otimes \Hom(M,N) \to \tau N.\end{equation}
We describe these maps explicitly. The map across the top of the triangle diagram \ref{dgm:triangle} is \[\tau \ev[M][N]: \tau M \otimes \Hom_{Q}(M,N) \to \tau N\] 
given by 
    \[ \left( \sum\limits_{a\in Q_1} \overline{p}_a\otimes m_a \right) \otimes \, f \mapsto \sum\limits_{a\in Q_1} \overline{p}_a \otimes \, f_{sa}(m_a) \] 
where \[f=(f_x)_{x\in Q_0} \in \Hom_{Q}(M,N), \quad  \overline{p}_a \in I(ta), \quad  m_a \in M_{sa}\] and extended linearly to $\tau M \subseteq \bigoplus_{a\in Q_1} I(ta) \otimes M_{sa}$. The map down the left side is
    \[\mathbbm{1}_{\tau M} \otimes \, \varphi: \tau M \otimes \Hom_{Q}(M,N) \to \tau M \otimes \Hom_{Q}(\tau M, \tau N)\] 
given by 
    \[ \left( \sum\limits_{a\in Q_1} \overline{p}_a\otimes m_a\right) \otimes \, f \mapsto \left( \sum\limits_{a\in Q_1} \overline{p}_a \otimes m_a \right)\otimes \tau f\] 
where \[\tau f = \bigoplus_{a\in Q_1} \mathbbm{1}_{I(ta)} \otimes \, f_{sa} \in \Hom_{Q}(\tau M, \tau N), \quad  \overline{p}_a \in I(ta), \quad m_a \in M_{sa}\] as proven in Proposition~\ref{prop:tau_f}$(b)$ and extended linearly to $\tau M$ as in the previous case.

The map moving from the bottom to the top right is
    \[\ev[\tau M][\tau N]: \tau M \otimes \Hom_{Q}(\tau M, \tau N) \to \tau N\] 
given by 
    \[ \left( \sum\limits_{a\in Q_1} \overline{p}_a\otimes m_a\right) \otimes \tau f \mapsto \left( \sum\limits_{a\in Q_1} \tau f(\overline{p}_a \otimes m_a) \right)\] 
Therefore, the composition $\ev[\tau M][\tau N] \circ (\mathbbm{1}_{\tau M} \otimes \varphi)$ gives 
    \begin{nalign*}
        \left( \sum\limits_{a\in Q_1} \overline{p}_a \otimes m_a \right) \otimes f \mapsto \left( \sum\limits_{a\in Q_1} \overline{p}_a \otimes m_a\right) \otimes \tau f  &\mapsto \sum\limits_{a\in Q_1} \tau f(\overline{p}_a \otimes m_a)\\
                    &= \sum\limits_{a\in Q_1} (\mathbbm{1}_{I(ta)} \otimes f_{sa})(\overline{p}_a\otimes m_a)\\
                    &= \sum\limits_{a\in Q_1} \overline{p}_a \otimes f_{sa}(m_a)
    \end{nalign*}
which is equal to the first map, $\tau \ev[M][N]$.
\end{proof}

\begin{corollary}\label{cor:ker_ev_t}
    For two representations $M, N \in \rep_{k}(Q)$ for some quiver $Q$ such that $\tau^{k-1} M, \tau^{k-1} N$ are not projective and for any $k \in \Z_{\geq 0}$, we have 
    \begin{equation} 
        \ker \ev[M][N] \cong \tau^{-k} \ker \ev[\tau^k M][\tau^k N].
    \end{equation}
    Similarly, the statement holds for $k < 0$ if $\tau^{k+1} M, \tau^{k+1} N$ are not injective.
\end{corollary}

\begin{proof}
    Suppose $k=1$. It suffices to take $M, N \in \rep_{k} Q \subseteq D^b( \rep_{k} Q)$ non-projective. Then the result follows from Theorem~\ref{thm:tev_description} which states that $$ \ev[\tau M][\tau N] \circ (\mathbbm{1}_{\tau M} \otimes \, \varphi) = \tau \ev[M][N].$$ Because the maps $\mathbbm{1}_{\tau M}$ and $\varphi$ are both isomorphisms, we have that in $\rep_{k} Q$
        \begin{nalign}
            \ker \tau \ev[M][N] &= \ker \left(\ev[\tau M][\tau N] \circ (\mathbbm{1}_{\tau M} \otimes \, \varphi) \right) \nonumber\\
                                &\cong \ker \ev[\tau M][\tau N] .\label{eq:evtau1}
        \end{nalign}    
    Now consider the following short exact sequence in $\rep_{k} Q \subseteq D^b(\rep_{k}Q)$.
        \begin{equation}
            0\to\ker \ev[M][N] \to M\otimes \Hom(M,N)\xto{\ev[M][N]} \im \ev[M][N] \to 0
        \end{equation}
    This corresponds to a triangle 
        \begin{equation}\ker \ev[M][N] \to M\otimes \Hom(M,N)\xto{\ev[M][N]} \im \ev[M][N] \to (\ker \ev[M][N])[1]
    \end{equation}
    in $D^b(\rep_{k} Q)$. Then, since the Auslander-Reiten translation $\tau$ is an autoequivalence on $D^b(\rep_{k} Q)$ which preserves triangles (Lemma~\ref{lem:tau_autoequiv}), we have that the following is also a  triangle.
        \begin{equation}
            \tau \ker \ev[M][N] \to \tau M\otimes \Hom(M,N)\xto{\ev[M][N]}\tau \im \ev[M][N] \to \tau(\ker \ev[M][N])[1]
        \end{equation}
    and this triangle corresponds to the exact sequence
        \begin{equation}
            0\to \tau \ker \ev[M][N] \to \tau M\otimes \Hom(M,N)\xto{\tau \ev[M][N]} \tau \im \ev[M][N] \to 0.
        \end{equation}
    From this sequence, we conclude that 
        \begin{equation}
            \tau \ker \ev[M][N] \cong \ker \tau \ev[M][N].
        \end{equation} 
    Thus, by the isomorphism~\ref{eq:evtau1}, 
        \begin{equation} \tau \ker \ev[M][N] \cong \ker      \ev[\tau M][\tau N].
        \end{equation}
    Next, we can apply the inverse $\tau^{-1}$ to get 
        \begin{equation}
            \ker \ev[M][N] \cong \tau^{-1} \ker \ev[\tau M][\tau N].
        \end{equation}
    A similar argument works for $\tau^k$ for any power $k \in \Z$, as long as $\tau^{k-1}M, \tau^{k-1}N$ are not projective if $k$ is positive and $\tau^{k-1}M, \tau^{k-1}N$ are not injective if $k$ is negative.
\end{proof}

 Finally, notice that the evaluation morphism is additive in the first and second arguments under a few restrictions. 

\begin{proposition}
    Given two finite-dimensional representations $M \cong \bigoplus_{i=1}^n M_i$ and $N \cong \bigoplus_{i=1}^n N_i$ in $\rep_k Q$, the evaluation morphism
    \begin{equation}
        \ev[M][N] : M \otimes_k \Hom_Q(M, N) \longrightarrow N
    \end{equation}
    splits as a direct sum over the source:
    \[
        M \otimes_k \Hom_Q(M, N) \cong \bigoplus_{i=1}^n \left( M_i \otimes_k \Hom_Q(M_i, N) \right),
    \]
    and if, in addition,
    \[
        \Hom_Q(M_i, N_j) = 0 \quad \text{for all } i \ne j,
    \]
    then the map also splits over the target, and we have
    \begin{equation}
        \ev[M][N] = \bigoplus_{i=1}^n \ev[M_i][N_i] :
        \bigoplus_{i=1}^n \left( M_i \otimes_k \Hom_{Q}(M_i, N_i) \right)
        \longrightarrow \bigoplus_{i=1}^n N_i.
    \end{equation}
\end{proposition}

\begin{proof}
    Using additivity of the $\Hom$-functor and the tensor product $\otimes$, we have isomorphisms:
    \begin{align}
        \Hom_Q(M, N) &\cong \Hom_Q\left(\bigoplus_i M_i, N \right) \cong \prod_i \Hom_Q(M_i, N) \\
        M \otimes_k \Hom_Q(M, N) &\cong \left(\bigoplus_i M_i\right) \otimes_k \left(\prod_i \Hom_Q(M_i, N)\right)
    \end{align}

    Because all representations are finite-dimensional, the product is naturally isomorphic to the direct sum
        \[\prod_i \Hom_Q(M_i, N) \cong \bigoplus_i \Hom_Q(M_i, N).\]
        
    and thus
        \[M \otimes_k \Hom_Q(M, N) \cong \bigoplus_i \left( M_i \otimes_k \Hom_Q(M_i, N) \right).\]

    For each $i$, define the $i$-th evaluation morphism
        \[\ev[M_i][N] : M_i \otimes_k \Hom_Q(M_i, N) \longrightarrow N, \quad m_i \otimes f_i \mapsto f_i(m_i).\]

    The evaluation morphism
        \[\ev[M][N] = \sum_i \ev[M_i][N] : \bigoplus_i \left( M_i \otimes_k \Hom_Q(M_i, N) \right) \longrightarrow N\]
        
    is then the direct sum of these component morphisms. Finally, because $\Hom_Q(M_i,N_i)=0$ for $i\ne j$, the evaluation morphism splits over the target as well and we get
        \begin{equation}
            \ev[M][N] : \bigoplus_i \left( M_i \otimes_k \Hom_Q(M_i, N_i) \right) \longrightarrow \bigoplus_i N_i.
        \end{equation}
        
    Therefore,
        \begin{equation}
            \ev[M][N] = \bigoplus_i \ev[M_i][N_i].
        \end{equation}

\end{proof}

\section{Completely-decomposable subcategories for the Kronecker quiver}\label{sec:pattern-theorems}

 The main theorems of this section are formulas which describe the kernel of the evaluation morphism $\ev[M][N]$ for $M, N$ indecomposable representations of the Kronecker quiver $\mathsf{K}$. The formulas give us information on when the kernel of the evaluation morphism is contained within a component AR($\mathsf{K}$) and when it is a direct sum of indecomposables from different components. This will be useful information in determining the structure of completely-decomposable subcategories for $\rep_{k} \mathsf{K}$. The nontrivial kernels of the evaluation morphism applied to indecomposable representations of the Kronecker quiver are summarized in the Table \ref{table:patterns}. 

\begin{table}[htp]
\centering
\caption{Nontrivial Kernels of the Evaluation Morphism for $Q = \mathsf{K}$}
\label{table:patterns}
\def\arraystretch{1.5}
    \begin{tabular}{|c|c|c|c|}
        \hline
        \textbf{M} & \textbf{N} & $\ker(\ev[M][N])$ & \textbf{Conditions} \\ \hline \hline
        {$P_m$} & {$P_n$} & $(P_{m-1})^{n-m}$ & $n > m \geq 1$ \\ \hline
        {$P_m$} & {$R_n(\lambda)$} & $(P_{m-1})^n$ & $m \geq 1$ \\  \hline
        {$P_m$} & {$I_n$} & $(P_{m-1})^{m+n+1}$ & $m \geq 1$ \\ \hline
        \multirow{2}{*}{$R_m(\lambda)$} & \multirow{2}{*}{$R_n(\lambda)$} & $(R_m(\lambda))^{m-1}$ & $m \leq n$ \\ \cline{3-4} 
         &  & $R_{m-n}(\lambda) \oplus (R_m(\lambda))^{n-1}$ & $m>n$ \\ \hline
        \multirow{2}{*}{$R_m(\lambda)$} & \multirow{2}{*}{$I_n$} & $(R_m(\lambda))^{m-1}$ & $m < n+1$ \\ \cline{3-4} 
         &  & $P_{m-n-1} \oplus (R_m(\lambda))^{m-1}$ & $m \geq n+1$\\ \hline
        {$I_m$} & {$I_n$} & $(I_{m+1})^{m-n}$ & $m > n \geq 1$ \\ \hline
    \end{tabular}
\end{table}

We now state the results pertaining to the kernel of the evaluation morphism between indecomposable representations of the Kronecker quiver $\mathsf{K}$. The statements, collectively referred to as the ``Pattern Theorems", are provided below and their proofs are deferred to Appendix~\ref{appendix:pattern-proofs}.

\begin{theorem}\label{thm:ev(Pm,N)}
    Given a postprojective indecomposable representation $P_m$, $m\geq 0$ and an indecomposable representation $N$ of the Kronecker quiver $\mathsf{K}$, we have that 
        \begin{enumerate}[(a)] 
            \item $\ker \ev[P_m][P_n] \cong (P_{m-1})^{n-m}, \quad n > m \geq 1 $,
                
            \item $\ker \ev[P_m][R_n(\lambda)] \cong (P_{m-1})^n, \quad m \geq 1$,
                
            \item $\ker \ev[P_m][I_n] \cong (P_{m-1})^{m+n+1}, \quad m \geq 1.$
        \end{enumerate}
    All other kernels are equal to 0.
\end{theorem}

\begin{theorem}\label{thm:ev(Rm,Rn)}
    Given two regular indecomposable representations $R_m(\lambda), R_n(\lambda)$, $m, n\geq 1, \lambda \in k \cup \{\infty\}$, of the Kronecker quiver $\mathsf{K}$, we have that 
        \[\ker \ev[R_m(\lambda)][R_n(\lambda)] \cong \begin{cases}(R_m(\lambda))^{m-1}, & m \leq n\\ R_{m-n}(\lambda) \oplus (R_m(\lambda))^{n-1}, & m>n.\end{cases} \]
    All other kernels are equal to 0.
\end{theorem}

\begin{theorem}\label{thm:ev(Rm,In)}
    Given a regular indecomposable representation $R_m(\lambda)$, $m \geq 1, \lambda \in k \cup \{\infty\}$, and a preinjective indecomposable representation $I_n, n \geq 0$, of the Kronecker quiver $\mathsf{K}$, we have that 
        \[\ker \ev[R_m(\lambda)][I_n] \cong \begin{cases} (R_m(\lambda))^{m-1},& m < n+1 \\ P_{m-n-1} \oplus (R_m(\lambda))^{m-1},& m \geq n+1\end{cases}\]
    All other kernels are equal to 0.
\end{theorem}

\begin{theorem}\label{thm:ev(Im,In)}
    Given two preinjective indecomposable representations $I_m, I_n, m,n \geq 0$, of the Kronecker quiver $\mathsf{K}$, we have that 
        \[\ker \ev[I_m][I_n] \cong (I_{m+1})^{m-n}, \quad m > n \geq 1.\]
    All other kernels are equal to 0.
\end{theorem}

It is clear from these theorems that certain subcategories satisfy the conditions of Theorem~\ref{cor:specialC}. For example, the full subcategory of all postprojectives $\{P_n\}_{n \in \Z}$, the full subcategory of all postprojectives $\{I_n\}_{n \in \Z}$, and the full subcategory of all postprojectives $\{R_n(\lambda)\}_{n \in \Z, \lambda \in k}$.

The following is a conjecture which supposes the existence of one class of examples of subcategories of $\rep_{k} \mathsf{K}$ that satisfy the conditions of the theorem. 

\begin{proposition}\label{conj:specialC_Kronecker}
    Let $\mathcal{X} \subseteq k \cup \{\infty\}$ be any subset. Define $\mathcal{C}_{\mathcal{X}} \subseteq \rep_k \mathsf{K}$ to be the full additive subcategory generated by the indecomposable representations 
        \[\{P_n,\, I_n,\, R_n(\lambda) \mid n \in \Z_{\geq 0},\ \lambda \in \mathcal{X}\},\]
    
    that is,
        \[\mathcal{C}_{\mathcal{X}} = \operatorname{add}\left(\{P_n,\, I_n,\, R_n(\lambda) \mid n \in \Z_{\geq 0},\ \lambda \in \mathcal{X}\}\right).\]
        
    Then \(\mathcal{C}_{\mathcal{X}}\) satisfies the hypotheses of Theorem~\ref{cor:specialC}.
\end{proposition}

\begin{proof}
    By construction, the category $\mathcal{C}_{\mathcal{X}}$ is a subcategory of $\rep_k \mathsf{K}$ that is closed under direct sums. It only remains to show that $\ker(\ev[M][N]) \in \mathcal{C}_{\mathcal{X}}$ for all pairs of objects $M, N \in \mathcal{C}_{\mathcal{X}}$.\\

    Let $M, N \in \mathcal{C}_{\mathcal{X}}$. Then, by definition, $M \cong \oplus_{i=1}^n M_i$ and $N \cong \oplus_{i=1}^n N_i$ where $M_i, N_i \in \{P_n,\, I_n,\, R_n(\lambda) \mid n \in \Z_{\geq 0},\ \lambda \in \mathcal{X}\}$. \\

    The subcategory $\mathcal{C}_{\mathcal{X}}$ is by definition the full additive subcategory generated by all $P_n$, $I_n$, and those regular indecomposables $R_n(\lambda)$ for which $\lambda \in \mathcal{X} \subseteq k \cup \{\infty\}$.

    Our goal is to show that $\mathcal{C}_{\mathcal{X}}$ is closed under taking kernels of morphisms between its objects.\\
    
    It is known that the category $\rep_k \mathsf{K}$ of finite-dimensional representations of $\mathsf{K}$ is hereditary. Thus, subrepresentations  of representations in $\rep_k \mathsf{K}$ are objects of $\rep_k \mathsf{K}$. We know that kernels of morphisms are subrepresentations, so $\ker \ev[M][N] \in \rep_k \mathsf{K}$ for every $M, N \in \mathcal{C}_{\mathcal{X}}$. Furthermore, because $\mathsf{K}$ is also Krull-Schmidt, the kernel decomposes into a direct sum of indecomposable representations.
    
    The postprojective components $\{P_n\}$ and preinjective components $\{I_n\}$ are well-known to be closed under taking submodules. This means that any subrepresentation of a postprojective representation is again postprojective, and similarly for preinjective representations. Hence, any kernel of a morphism between sums of $P_n$ and $I_n$ lies in the additive closure of these indecomposables. This additive closure is clearly contained in $\mathcal{C}_{\mathcal{X}}$.
    
    The kernel of the evaluation morphism between two regular indecomposable representations  $R_m(\lambda)$ and $R_n(\lambda)$ is nonzero exactly when they have the same parameter $\lambda$. Thus, if we restrict to the subcategory $\mathcal{C}_{\mathcal{X}}$ which only includes regular representations with parameters $\lambda \in \mathcal{X}$, the kernel of the evaluation morphism between objects in $\mathcal{C}_{\mathcal{X}}$ must decompose into a direct sum of indecomposables $R_m(\lambda)$ with $\lambda \in \mathcal{X}$. No regular summands with parameters outside $\mathcal{X}$ can appear in kernels of morphisms within $\mathcal{C}_{\mathcal{X}}$.
    
    Now, any object in $\mathcal{C}_{\mathcal{X}}$ is a finite direct sum of indecomposables of the three types described above, with regular summands restricted to parameters $\lambda$ in $\mathcal{X}$. Since kernels are subrepresentations, the kernel of a morphism between two such objects decomposes into a direct sum of indecomposables, each coming from one of these three classes and respecting the parameter restrictions on regular summands.\\
    
    Hence, the kernel remains in $\mathcal{C}_{\mathcal{X}}$.
\end{proof}

The reader may refer to either \cite{ARS} or \cite{ASS2} for extensive details on the properties of the indecomposable representations of $\mathsf{K}$.

\begin{remark}
    We are particularly interested in the case where $\mathcal{X}=\{0, 1, \infty\}$ because in that setting, all morphisms in the subcategory $\mathcal{C}_{\mathcal{X}}$ can be presented by matrices using only 0s and 1s as entries, as is often the case for persistence modules.
\end{remark}

Our analysis of the Kronecker quiver demonstrates that various naturally occurring subcategories satisfy the hypotheses of our refined theorem. The behavior of the kernel of the evaluation morphism and its compatibility with the Auslander-Reiten translation suggest a robust framework for detecting isomorphism classes within completely decomposable subcategories. These results offer concrete examples and provide evidence for the applicability of our results to other types of quivers.

\clearpage
\appendix
\section{Proofs of Pattern Theorems}
\label{appendix:pattern-proofs}
This appendix contains detailed proofs of the Pattern Theorems stated in Section~\ref{sec:pattern-theorems}.
\subsection{Helpful Lemmas}
We begin here by collecting some computations as lemmas which we will be necessary in order to provide the Pattern Theorems. The lemmas here will be useful for all the proofs in this section and lemmas needed for a particular subsection will be stated and proved at the end of each subsection. Whenever necessary, computations are done in the derived category $\rep_{k} \mathsf{K}$ by embedding $\rep_{k} \mathsf{K}$ into $D^b(\rep_{k} \mathsf{K})$ as described in Section~\ref{sec:derivedcateg}.

The next lemmas from \cite{Barot} will provide justification for the choices of $m, n$ in Table~\ref{table:patterns}. The domain of the evaluation involves a tensor with the Hom-space $\Hom(M,N)$. When this Hom-space is $0$, the evaluation morphism becomes the map $0 \to N$ and thus has an uninteresting kernel. These lemmas state that the morphisms in the AR quiver for $\mathsf{K}$ can only go in left to right in the ordering of indecomposable components $\mathcal{P} \rightarrow \mathcal{R} \rightarrow \mathcal{I}$.

\begin{lemma}
If $M$ is an indecomposable representation of $\mathsf{K}$ such that there exists a non-zero morphism $f: M \to P_n$, then $M$ is isomorphic to $P_m$ for some $m \leq n$. If $m \leq n$, the space $\Hom_Q(P_m, P_n)$ has dimension $n-m+1$.
\end{lemma}

\begin{lemma}
If $N$ is an indecomposable representation of $\mathsf{K}$ such that there exists a non-zero morphism $f: I_m \to N$, then $N$ is isomorphic to $I_n$ for some $n \leq m$. If $n \leq m$, the space $\Hom(I_m, I_n)$ has dimension $m-n+1$.
\end{lemma}

\begin{lemma}\label{lem:lamba=mu}
If $f:R_m(\lambda) \to R_n(\mu)$ is a non-zero morphism for $\lambda, \mu \in k \cup \{\infty\}$, then $\lambda = \mu$. The $\Hom$-space $\Hom_Q(R_m(\lambda), R_n(\lambda))$ has dimension $\min(m,n)$.
\end{lemma}

The proof of these facts provided in \cite{Barot} give bases for $\Hom$-spaces between indecomposable representations of $\mathsf{K}$ which we summarize in the following lemma:

\begin{lemma}\label{lem:Hom_mats}
The $\Hom$-spaces $\Hom(I_m,I_n)$ and $\Hom(R_m(\lambda),R_n(\lambda))$ are as described below.
    \begin{enumerate}[(a)]
        \item $\Hom(I_m,I_n)$, $m \geq n$, is given by the following pairs of matrices $(A,B)$:
                \begin{nalign} 
                    A &=
                        \left[
                        \begin{array}{ccccccc}
                            a_1 & a_2 & \cdots & a_{m-n+1}    & & & \\
                                & a_1 & a_2    & \cdots & a_{m-n+1} & &  \\
                     &     & \ddots & \ddots &      & \ddots & \\
                                &     &        & a_1    & a_2  & \cdots & a_{m-n+1} \\
                        \end{array}
                        \right] \in k^{(n+1) \times (m+1)} \\
        \vspace{10pt}\\               
                    B &=
                        \left[
                        \begin{array}{ccccccc}
                            a_1 & a_2 & \cdots & a_{m-n+1} & & & \\
                                & a_1 & a_2    & \cdots & a_{m-n+1} & &  \\
                     &     & \ddots & \ddots &      & \ddots & \\
                                &     &        & a_1    & a_2  & \cdots & a_{m-n+1} \\
                        \end{array}
                        \right] \in k^{n \times m}
                \end{nalign}  

        \item $\Hom(R_m(\lambda), R_n(\lambda))$ is given by the following pairs of matrices $(A,B)$:
            \begin{enumerate}[i.]
                \item If $m \leq n$,
                    \[A=B=
                        \left[
                        \begin{array}{ccccc}
                            a_1 & a_2      & \cdots & a_{m}\\
                                & \ddots   & \ddots & \vdots \\
                                &          & \ddots & a_2    \\
                                &        &  & a_1 \\
                                &          &        & \\
                                &          &        & 
                        \end{array}
                        \right] \in k^{n \times m}
                    \]  
                \item If $m > n$,
                    \[A = B =
                        \left[
                        \begin{array}{ccccccc}
                          & & a_1 & a_2      & \cdots & a_{n}\\
                         &  &     & \ddots   & \ddots & \vdots \\
                       &  &  &  & \ddots & a_2  \\
                         & &  &  &        & a_1    
                        \end{array}
                        \right] \in k^{n\times m}
                    \] 
            \end{enumerate}
    \end{enumerate}
\end{lemma}

Finally, we reference \cite{BSaffinealg2012} for a description of the $\AR$ quiver of the derived category $D^b(\rep{\mathsf{K}})$ for the Kronecker quiver $\mathsf{K}$. In particular, we note that in $D^b(\rep{\mathsf{K}})$, we have $\tau^{-1} I_0[n-1] = P_1[n]$. A diagram is given in Figure~\ref{fig:AR(Db(K))} below.

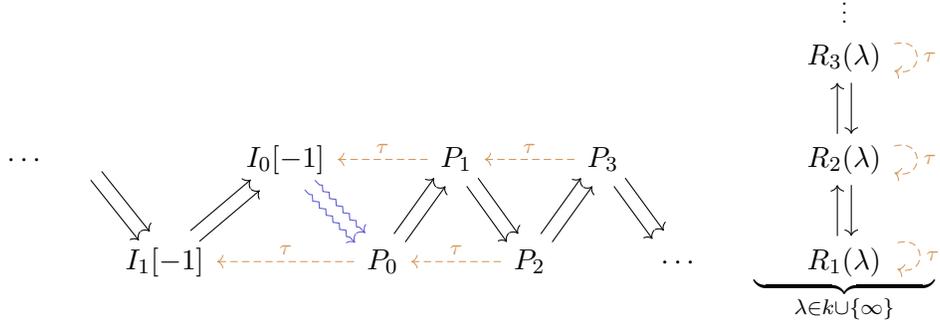
\begin{figure}[H]
\caption{Auslander-Reiten quiver of $D^b(\rep{\mathsf{K}})$}
\[\begin{tikzcd}[ampersand replacement=\&, column sep=tiny]
	\&\&\&\&\&\&\&\&\&\& {} \\
	\&\&\&\&\&\&\&\&\&\& {R_3(\lambda)} \arrow["{\tau }"' swap,loop right, dashed, color={rgb,255:red,214;green,153;blue,92},looseness=3] \\
	{\cdots} \& {} \&\& {I_0[-1]} \&\& {P_1} \& {} \& {P_3} \& {} \&\& {R_2(\lambda)} \arrow["{\tau }"' swap,loop right, dashed, color={rgb,255:red,214;green,153;blue,92},looseness=3] \\
	\& {} \& {I_1[-1]} \&\& {P_0} \&\& {P_2} \&\& {\cdots} \&\& {\underbrace{\phantom{++;} R_1(\lambda) \phantom{++;}}_{\lambda \in k \cup \{\infty\}}} \arrow["{\tau }"' swap,loop right, dashed, color={rgb,255:red,214;green,153;blue,92},looseness=3,start anchor={[xshift=-4ex,yshift=2.5mm]east}, end anchor={[xshift=-4ex, yshift=-1mm]east}] \\
	\arrow[shift left=1, from=4-5, to=3-6]
	\arrow[shift left=1, from=3-6, to=4-7]
	\arrow[shift left=1, from=4-7, to=3-8]
	\arrow[shift left=1, from=4-3, to=3-4]
	\arrow[shift right=1, from=4-5, to=3-6]
	\arrow[shift right=1, from=3-6, to=4-7]
	\arrow[shift right=1, from=4-7, to=3-8]
	\arrow[shift right=1, from=4-3, to=3-4]
	\arrow[shift left=1, from=3-11, to=2-11]
	\arrow[shift left=1, from=3-11, to=4-11]
	\arrow["\tau"', color={rgb,255:red,214;green,153;blue,92}, dashed, from=3-8, to=3-6]
	\arrow["\tau"', color={rgb,255:red,214;green,153;blue,92}, dashed, from=4-7, to=4-5]
	\arrow["\vdots"{description}, draw=none, from=2-11, to=1-11]
	\arrow[shift left=1, color={rgb,255:red,92;green,92;blue,214}, squiggly, from=3-4, to=4-5]
	\arrow[shift right=1, color={rgb,255:red,92;green,92;blue,214}, squiggly, from=3-4, to=4-5]
	\arrow["{\tau }"', color={rgb,255:red,214;green,153;blue,92}, dashed, from=4-5, to=4-3]
	\arrow["\tau"', color={rgb,255:red,214;green,153;blue,92}, dashed, from=3-6, to=3-4]
	\arrow[shift left=1, from=4-11, to=3-11]
	\arrow[shift left=1, from=2-11, to=3-11]
	\arrow[shorten >=1ex, shift left=1, from=3-8, to=4-9]
	\arrow[shorten >=1ex, shift right=1, from=3-8, to=4-9]
	\arrow[shift right=1, from=3-2, to=4-3]
	\arrow[shift left=1, from=3-2, to=4-3]
\end{tikzcd}\]
\label{fig:AR(Db(K))}
\end{figure}

The next lemma will present alternative criteria to those of the recognition theorem for direct sums of vector spaces. Recall that the recognition theorem tells us that if we have subspaces $S_1,\ldots,S_n$ of a vector space $V$, then we may conclude that $V= \oplus_{i=1}^n S_i$ if $\sum_{i}^n S_i = V$ and $S_i \cap \sum_{j\neq i} S_j = \{0\}$ for each $i$.
\begin{lemma}\label{lem:recognition}
    Given subspaces $S_1,\ldots,S_l$ of a vector space $V$, let $\mathcal{B}_i$ be a basis of $S_i$ for $i=1,\ldots,n$. In addition, suppose that
        \begin{enumerate}[(a)]
            \item $\mathcal{B} := \bigcup\limits_{i=1}^n \mathcal{B}_i$ is a linearly independent set of elements of $V$ and
            \item $\sum\limits_{i=1}^n \dim S_i = \dim V$.
        \end{enumerate}
    Then, $V= \oplus_{i=1}^n S_i$.
\end{lemma}

\begin{proof}
    We will show that the conditions $(a)$ and $(b)$ of the lemma imply that $\sum_{i=1}^n S_i = V$ and $S_i \cap \sum_{j\neq i} S_j = \{0\}$ for each $i$. From there, it will follow, by the Recognition Theorem for direct products of groups, that $V= \oplus_{i=1}^n S_i$. Note that because we are only dealing with finite direct sums, the recognition theorem indeed applies. Also, because we are using vector spaces, the condition from the recognition theorem in the language of groups which requires commutativity is automatically satisfied. 

    Take $\mathcal{B}$ as defined in the lemma. Now suppose that there is an element $x \in S_i \cap \sum_{j\neq i} S_j$. Then, 
        \begin{nalign*}
                      x &= \sum_{k=1}^{\dim S_i} r_{k} \sigma_{ik} = \sum_{j\neq k} \sum_{k=1}^{\dim S_j} s_{jk} \sigma_{jk}\\
          \Rightarrow 0 &= \sum_{j\neq k} \sum_{k=1}^{\dim S_j} s_{jk} \sigma_{jk} - \sum_{k=1}^{\dim S_i} r_{k} \sigma_{ik}\\
                        &= \sum_{j=1}^n \sum_{k=1}^{\dim S_j} s_{jk} \sigma_{jk}
        \end{nalign*}
    where $s_{ik} = -r_k$ for $k=1,\ldots, \dim S_i$. Because the set $\{\sigma_{jk}\}_{j,k} = \mathcal{B}$ is linearly independent by part $(a)$ of the lemma, it must be that $s_{jk} = 0$ for all $j,k$. Then, $x=0$ and $S_i \cap \sum_{j\neq i} S_j = \{0\}$. This argument is independent of the fixed $i$, so we get the result for each $i=1,\ldots,n$. The $S_i$ are mutually disjoint and therefore so are their bases $\mathcal{B}_i$.

    Next, we want to show that $S := \sum_{i=1}^n S_i = V$. We know that $S \subseteq V$ because each $S_i \subseteq V$. Also, we may conclude that $\mathcal{B}$ is basis of $S$ with dimension
        \begin{equation}
            |\mathcal{B}| = \left|\bigsqcup\limits_{i=1}^n \mathcal{B}_i \right| = \sum\limits_{i=1}^n |\mathcal{B}_i| = \sum\limits_{i=1}^n \dim S_i
        \end{equation}
    By part $(b)$ of the lemma, $\sum_{i=1}^n \dim S_i = \dim V$, so it follows that $S$ is a submodule of $V$ with the same dimension. Hence, it must be that $S= \sum_{i=1}^n S_i =V$. 
\end{proof}

\begin{remark}
    Note that in the case of quiver representations, where there are vector spaces at each vertex and direct sums are defined vertex-wise, we can use the argument by vertex. In other words, we let $S_i = (S_i(x))_{x\in Q_0}$ so that when writing $S_i \subseteq V$, we mean $S_i(x) \subseteq V(x)$ for each vertex $x$. Similarly, we will use $\mathcal{B}_i = (\mathcal{B}_i(x))_{x \in Q_0}$ to refer to a collection of bases at each vertex. It will be understood that bases at different vertices do not interact and are linearly independent from one another. Finally, part $(b)$ of the lemma then becomes $\sum\limits_{i=1}^n \ddim S_i = \ddim V$.
\end{remark}

\subsection{From Postprojective}

\begin{customthm}{\ref{thm:ev(Pm,N)}}
    Given a postprojective indecomposable representation $P_m$, $m\geq 0$ and an indecomposable representation $N$ of the Kronecker quiver $\mathsf{K}$, we have that 
        \begin{enumerate}[(a)] 
            \item $\ker \ev[P_m][P_n] \cong (P_{m-1})^{n-m}, \quad n > m \geq 1 $,
                
            \item $\ker \ev[P_m][R_n(\lambda)] \cong (P_{m-1})^n, \quad m \geq 1$,
                
            \item $\ker \ev[P_m][I_n] \cong (P_{m-1})^{m+n+1}, \quad m \geq 1.$
        \end{enumerate}
    All other kernels are equal to 0.
\end{customthm}

\begin{proof} Consider the indecomposable representations $P_m$ and $N$ of the Kronecker quiver $\mathsf{K}$. Since $P_m$ is a postprojective element, there exists some $k \geq 1$ such that $\tau^k P_m$ is isomorphic to either $P_1 \cong P(x)$ if $m=2k+1$ or $P_0 \cong P(y)$ if $m=2k$. Also note that we have 
        \begin{equation}
            \Hom(P_m,N) \cong \Hom(\tau^k P_m, \tau^k N)
        \end{equation}
by Lemma~\ref{cor:tau_isos}$(a)$, which we may use to extend from $P_0$ or $P_1$ to the general $P_m$. Therefore, without loss of generality, we will consider the two cases $\ev[P_0][N]$ and $\ev[P_1][N]$. Note that Lemma~\ref{cor:tau_isos}$(a)$ requires that $N$ has no projective summands. Since $N$ is indecomposable, this means that we cannot apply it to the case where $N$ is $P_0$ or $P_1$. But, notice that if $N$ is $P_0$ or $P_1$, then $\Hom(P_m,N)$ is only nonzero when $m$ is 0 or 1. Since these are the base cases, we do not need the lemma to compute the kernel of the evaluation morphism and we can apply the lemma without issue.\\

\subsubsection{Case I:}
    Consider the evaluation morphism \(\ev[P_0][N]: P_0 \otimes \Hom(P_0, N) \to N\). By Lemma~\ref{lem:kerevP0}, we have that $\ker \ev[P_0][N] \cong I_0[-1] \otimes N(x)$ in the derived category. 
    
    By Corollary~\ref{cor:ker_ev_t}, we know that $\ker \ev[M][N] \cong \tau^{-k} \ker \ev[\tau^k M][\tau^k N]$ for $k \in \Z$, so we can use that $P_0 \cong \tau^k P_m$ to extend the kernel of the evaluation morphism from $M=P_0$ to $M=P_{m=2k}$:
        \begin{nalign}
            \ker \ev[P_{m}][N] &\cong \tau^{-k} \ker \ev[\tau^k P_m][\tau^k N]\\
                                  &\cong \tau^{-k} \ker \ev[P_0][\tau^k N]\\
                                  &\cong \tau^{-k} \left(I_0[-1] \otimes (\tau^k N)(x) \right)\\
                                  &\cong \tau^{-k}I_0[-1] \otimes (\tau^k N)(x)
        \end{nalign}    

    Note that $\tau^{-k}$ does not act on $(\tau^k N)(x)$. Recall from the shape of $\AR(D^b(\rep\mathsf{K}))$ (Figure~\ref{fig:AR(Db(K))}) that $\tau^{-1}I_0[-1] \cong P_1$, so that 
        \begin{nalign}
            \tau^{-k} I_0[-1] &\cong \tau^{-(k-1)} \tau^{-1} I_0[-1] \\
                              &\cong \tau^{-(k-1)} P_1 \\
                              &\cong P_{1+2(k-1)} \\
                              &\cong P_{2k-1}\\
                              &\cong P_{m-1}
        \end{nalign}
    Thus, we conclude that $\ker \ev[P_{m=2k}][N] \cong P_{m-1} \otimes (\tau^k N)(x)$. We now calculate $(\tau^k N)(x)$ for $N \cong P_n, R_n(\lambda), I_n$:
        \begin{enumerate}[(a)]
            \item $N = P_n \in \mathcal{P}, n \geq m$:
                \[(\tau^k P_n)(x) \cong P_{n-2k}(x) \cong k^{n-2k} = k^{n-m}\]

            \item $N = R_n(\lambda) \in \mathcal{R}$:
                \[(\tau^k R_n(\lambda))(x) \cong R_n(\lambda)(x) \cong k^n\]

            \item $N = I_n \in \mathcal{I}$:
                \[(\tau^k I_n)(x) \cong I_{n+2k}(x) \cong k^{n+2k+1} = k^{n+m+1}\]
        \end{enumerate}
    Finally, because $P_{m-1} \otimes k^l \cong (P_{m-1})^l$, we have shown the result.\\

\subsubsection{Case II:} 
    Now consider \(\ev[P_1][N]: P_1 \otimes \Hom(P_1, N) \to N\). By Lemma~\ref{lem:kerevP1}, 
    \begin{equation}
        \ker \ev[P_1][N] \cong P_0^{\dim \ker \begin{bsmallmatrix} \alpha_N & \beta_N \end{bsmallmatrix}}
    \end{equation}
    
    We calculate the dimension of the kernel of the evaluation morphism for each choice of $N=P_N, I_n, R_n(\lambda)$. As noted above, it reduces to finding $\dim \ker \begin{bsmallmatrix} \alpha_N & \beta_N \end{bsmallmatrix}$, the dimension of the kernel of the block matrix $\begin{bsmallmatrix} \alpha_N & \beta_N \end{bsmallmatrix}$.

    For $N=P_n \in \mathcal{P}$, we have 
    \begin{equation}\begin{bmatrix}\alpha_N & \beta_N\end{bmatrix} =   \begin{bmatrix} 
                                                     & 0 \, \ldots \, 0 \\
                                          \mathbbm{1}_n       &                  \\
                                                     &       \mathbbm{1}_n       \\
                                    0 \, \ldots \, 0 &     
                                \end{bmatrix}\end{equation}
    which results in an $(n+1)\times(2n)$ matrix where the two copies of $\mathbbm{1}_n$ overlap for $n-2$ rows (all rows excluding the rows containing $0 \ldots 0$ at the bottom or the top of each of the copies of $\mathbbm{1}_n$). Because each non-zero entry is in a row with no other elements, we may use the columns from the left copy of $\mathbbm{1}_n$ to cancel all but the final column in the right copy. We obtain an equivalent matrix 
                \begin{equation}
                \begin{bmatrix} 
                          \mathbbm{1}_n       & 0_{1\times n} \\
                     0_{1\times n}   &     \mathbbm{1}_n
                \end{bmatrix}
                \longleftrightarrow
                \begin{bmatrix} 
                       \mathbbm{1}_n       & 0_{n\times (n-1)} & 0_{n \times 1}\\
                    0_{1\times n} & 0 \, \ldots \, 0  & 1
                \end{bmatrix}
                \end{equation}
    where $0_{k \times l}$ indicates submatrix of $k \times l$ zeroes. Notice that there is a remaining entry $1$ at the bottom right so that the rank of this matrix is $n+1$ and its kernel has size $2n-(n+1)=n-1$ corresponding to the $n-1$ columns of all zeros. Therefore, $\dim \ker \begin{bsmallmatrix} \alpha_{P_n} & \beta_{P_n} \end{bsmallmatrix} = n-1$ and $\ker \ev[P_1][P_n]\cong (P_0)^{n-1}$.

    For $N=R_n(\lambda)$, we have an $n\times 2n$ matrix
                \begin{equation}\begin{bmatrix}\alpha_N & \beta_N\end{bmatrix} = \begin{bmatrix} \mathbbm{1}_n & J_n(\lambda) \end{bmatrix} \longleftrightarrow \begin{bmatrix} \mathbbm{1}_n & 0_{n\times n}\end{bmatrix}\end{equation}
    In this case, we notice that we may use the columns of the identity matrix to completely cancel the Jordan block $J_n(\lambda)$. We can easily see that the kernel will have dimension $n$ and conclude $\ker \ev[P_1][R_n(\lambda)]\cong (P_0)^n$.

    For $N=I_n$, we obtain a $n \times (2n+2)$ matrix 
                \begin{equation}\begin{bmatrix} \mathbbm{1}_n & 0 & 0 & \mathbbm{1}_n\end{bmatrix} \longleftrightarrow \begin{bmatrix} \mathbbm{1}_n & 0_{n\times 1} & 0_{n\times 1} & 0_{n\times n}\end{bmatrix}\end{equation} 
    which has dimension of the kernel equal to $n+2$. So, $\ker \ev[P_1][I_n]\cong (P_0)^{n+2}$.

    Now, by using Corollary~\ref{cor:ker_ev_t}, we can get the kernel of $\ev[P_m][N]$, with $m=2k+1$ odd. If $\tau^k P_m \cong P_1$, then:
        \begin{nalign}
            \tau^{-k} \ker \ev[P_{m=2k+1}][N] &\cong \tau^{-k} \ker \ev[\tau^k P_m][\tau^k N]\\
                                              &\cong \tau^{-k} \ker \ev[P_1][\tau^k N]
        \end{nalign}
    Then, 
        \begin{enumerate}[(1)]
            \item For each $N = P_n \in \mathcal{P}:$
                \begin{nalign}
                    \ker \ev[P_1][P_n] &\cong \tau^{-k} \ker \ev[P_1][\tau^k P_n]\\
                                       &\cong \tau^{-k} \ker \ev[P_1][P_{n-2k}]\\
                                       &\cong \tau^{-k} (P_0)^{n-2k-1}\\
                                       &\cong (P_{2k})^{n-2k-1}\\
                                       &\cong (P_{m-1})^{n-m}
                \end{nalign}

            \item For each $N = R_n(\lambda) \in \mathcal{R}:$
                \begin{nalign}
                    \ker \ev[P_1][R_n(\lambda)] &\cong \tau^{-k} \ker \ev[P_1][\tau^k R_n(\lambda)]\\
                                       &\cong \tau^{-k} \ker \ev[P_1][R_n(\lambda)]\\
                                       &\cong \tau^{-k} (P_0)^n\\
                                       &\cong (P_{2k})^n\\
                                       &\cong (P_{m-1})^n
                \end{nalign}

            \item For each $N = I_n \in \mathcal{I}:$
                \begin{nalign}
                    \ker \ev[P_1][I_n] &\cong \tau^{-k} \ker \ev[P_1][\tau^k I_n]\\
                                       &\cong \tau^{-k} \ker \ev[P_1][I_{n+2k}]\\
                                       &\cong \tau^{-k} (P_0)^{n+2k+2}\\
                                       &\cong (P_{2k})^{n+2k+2}\\
                                       &\cong (P_{m-1})^{n+m+1}
                \end{nalign}
        \end{enumerate}
        These results agree with those from Case I and combine to yield the desired formulas for $\ev[P_m][N]$.
\end{proof}

\begin{lemma}\label{lem:kerevP0}
    Let $N$ be any indecomposable representation of the Kronecker quiver $\mathsf{K}$. Then,
        \[\ker \ev[P_0][N] \cong I_0[-1] \otimes N(x)\]
\end{lemma}

\begin{proof}
    Consider \[\ev[P_0][N]:P_0 \otimes \Hom(P_0,N) \to N.\] Because $P_0 \cong P(y)$, by Lemma~\ref{cor:tau_isos} we obtain, $\Hom(P_0,N) \cong \Hom(P(y),N) \cong N(y)$. We may also represent the representations $P_0$ and $N$ in the evaluation morphism as $P_0: 0\toto k$ and $N: N(x) \toto N(y)$. 
    Thus, we may rewrite $\ev[P_0][N]$ to be
        \[\ev[P_0][N]:(0 \toto k) \otimes N(y) \to N(x) \toto N(y)\]
    We can further simplify by applying the tensor in the first argument:
        \[\ev[P_0][N]:0 \toto N(y) \to N(x) \toto N(y)\]
    It is then clear that $\ev[P_0][N]$ is equivalent to a pair of maps $(A,B)$ such that the following diagram commutes:

    \begin{equation}\begin{tikzcd}
        0 && {N(x)} \\
        \\
        {N(y)} && {N(y)}
        \arrow["A", from=1-1, to=1-3]
        \arrow["B"', from=3-1, to=3-3]
        \arrow["0", shift left=2, from=1-1, to=3-1]
        \arrow["0"', shift right=2, from=1-1, to=3-1]
        \arrow["{\beta_N}", shift left=2, from=1-3, to=3-3]
        \arrow["{\alpha_N}"', shift right=2, from=1-3, to=3-3]
    \end{tikzcd}\end{equation}

    We must have $A=0$ and $B=\mathbbm{1}_{N(y)}$. As both of these maps are injective and have no kernel, we focus our attention on $\coker \ev[P_0][N]$: 
        \begin{nalign}
            \coker \ev[P_0][N] &\cong \frac{N}{\im \ev[P_0][N]}\\
                               &\cong \frac{N(x)}{\im 0} \toto \frac{N(y)}{\im \mathbbm{1}_{N(y)}} \\
                               &\cong N(x) \toto 0\\
                               &\cong (k \toto 0) \otimes N(x)\\
                               &\cong I_0 \otimes N(x)
        \end{nalign}
    We can fit $\coker \ev[P_0][N] \cong I_0 \otimes N(x)$ into the following short exact sequence:
        \[0 \to P_0 \otimes N(y) \xto{\ev[P_0][N]} N \to I_0 \otimes N(x) \to 0.\]
    Note that this sequence is split: 
        \begin{nalign}
            (P_0 \otimes N(y)) \oplus (I_0 \otimes N(x)) &\cong (0 \toto k) \otimes N(y) \oplus (k \toto 0) \otimes N(x) \\
            &\cong 0 \toto N(y) \oplus N(x) \toto 0\\
            &\cong N(x) \toto N(y)\\
            &\cong N
        \end{nalign}
    Therefore, by Lemma~\ref{def:triangles}, in the derived category $D^b(\rep\mathsf{K})$, we have a triangle
            \[P_0 \otimes N(y) \xto{\ev[P_0][N]} N \to I_0 \otimes N(x) \to (P_0 \otimes N(y))[1].\]
    As with Definition~\ref{def:cone}, we will take the $(I_0 \otimes N(x))[-1] \cong I_0[-1] \otimes N(x)$ to be the co-cone of the $\ev[P_0][N]$ and so it serves as a ``kernel" of the $\ev[P_0][N]$. We will write $\ker \ev[P_0][N] \cong I_0[-1] \otimes N(x)$ in the derived category.
\end{proof}

\begin{lemma}\label{lem:kerevP1}
    Let $N$ be any indecomposable representation of the Kronecker quiver $\mathsf{K}$. Then,
        \[\ker \ev[P_1][N] \cong P_0^{\dim \ker \begin{bsmallmatrix} \alpha_N & \beta_N \end{bsmallmatrix}}\]
\end{lemma}

\begin{proof}
    we begin by rewriting the $\Hom(P_1,N) \cong \Hom(P(x), N) \cong N(x)$ via Proposition~\ref{prop:hom_isos} to get that $\Hom(P_1,N) \cong \Hom(P(x),N) \cong N(x)$ and inserting the descriptions of the representations $P_1, N$:
        \[\ev[P_1][N]: (k \toto k^2) \otimes N(x) \to N(x) \toto N(y)\]
    which, after simplifying the tensor in the first argument, becomes
        \[\ev[P_1][N]: N(x) \toto N(x) \oplus N(x) \to N(x) \toto N(y).\]
    In this situation, the evaluation morphism is a pair of maps $(A,B)$ which make the following diagram commute:
\begin{equation}\begin{tikzcd}[ampersand replacement=\&]
    {N(x)} \& {} \& {N(x)} \\
    \\
    {N(x)\oplus N(x)} \& \& {N(y)}
    \arrow["A", from=1-1, to=1-3]
    \arrow["B"', from=3-1, to=3-3]
    \arrow["{\begin{bsmallmatrix} 0 \\ 1 \end{bsmallmatrix}}", shift left=2, from=1-1, to=3-1]
    \arrow["{\begin{bsmallmatrix} 1 \\ 0 \end{bsmallmatrix}}"', shift right=2, from=1-1, to=3-1]
    \arrow["{\beta_N}", shift left=2, from=1-3, to=3-3]
    \arrow["{\alpha_N}"', shift right=2, from=1-3, to=3-3]
\end{tikzcd}\end{equation}

    We take $A=id_{N(x)}$ and write the equations involving $B$: 
        \begin{nalign*}
            \alpha_N \, id_{N(x)} &= B \, \begin{bsmallmatrix} 1 \\ 0 \end{bsmallmatrix} \\
            \beta_N \, id_{N(x)} &= B \, \begin{bsmallmatrix} 0 \\ 1 \end{bsmallmatrix}
        \end{nalign*}
    Thus, it must be that the first column of $B$ is $\alpha_N$ and the second column is $\beta_N$. In other words, we have that $B = \begin{bsmallmatrix} \alpha_N & \beta_N \end{bsmallmatrix}$. So,
        \begin{nalign}
                 \ker \ev[P_1][N] &\cong \ker \mathbbm{1}_{N(x)} \toto \ker \begin{bsmallmatrix} \alpha_N & \beta_N \end{bsmallmatrix} \\
                 &\cong 0 \toto \ker \begin{bsmallmatrix} \alpha_N & \beta_N \end{bsmallmatrix} \\
                 &\cong (0 \toto k) \otimes \ker \begin{bsmallmatrix} \alpha_N & \beta_N \end{bsmallmatrix}\\
                 &\cong P_0 \otimes \ker \begin{bsmallmatrix} \alpha_N & \beta_N \end{bsmallmatrix} \\
                 &\cong P_0^{\dim \ker \begin{bsmallmatrix} \alpha_N & \beta_N \end{bsmallmatrix}} 
        \end{nalign}
\end{proof}

\subsection{From Regular to Regular}

\begin{customthm}{\ref{thm:ev(Rm,Rn)}}
    Given two regular indecomposable representations $R_m(\lambda), R_n(\lambda)$, $m, n\geq 1, \lambda \in k \cup \{\infty\}$, of the Kronecker quiver $\mathsf{K}$, we have that 
        \[\ker \ev[R_m(\lambda)][R_n(\lambda)] \cong \begin{cases}(R_m(\lambda))^{m-1}, & m \leq n\\ R_{m-n}(\lambda) \oplus (R_m(\lambda))^{n-1}, & m>n.\end{cases} \]
    All other kernels are equal to 0.
\end{customthm}

\begin{proof}
    First, we refer to Lemmas~\ref{lem:lamba=mu} and \ref{lem:Hom_mats} which state that a map between two regular indecomposables $R_m(\lambda)$ and $R_n(\mu)$, $\lambda, \mu \in k \cup \{\infty\}$, is only nonzero when $\lambda=\mu$. We also have that the $\Hom$-space between two such representations is given by the following pairs of matrices $(A,B)$:
        \begin{enumerate}[]
            \item If $m \leq n$,
                \[A=B=
                    \left[
                    \begin{array}{ccccc}
                        a_1 & a_2      & \cdots & a_{m}\\
                            & \ddots   & \ddots & \vdots \\
                            &          & \ddots & a_2    \\
                            &        &  & a_1 \\
                            &          &        & \\
                            &          &        & 
                    \end{array}
                    \right] \in k^{n \times m}
                \]  
            \item If $m > n$,
                \[A = B =
                    \left[
                    \begin{array}{ccccccc}
                      & & a_1 & a_2      & \cdots & a_{n}\\
                     &  &     & \ddots   & \ddots & \vdots \\
                   &  &  &  & \ddots & a_2  \\
                     & &  &  &        & a_1    
                    \end{array}
                    \right] \in k^{n\times m}
                \]  
        for $a_i \in k$. $\Hom_Q(R_m(\lambda), R_n(\lambda))$ has dimension $\min(m,n)$.
        \end{enumerate}
    Note that the structure of each $\Hom$-space does not vary with $\lambda$. Without loss of generality, we can therefore consider $\lambda=0$. Let $R_m := R_m(0)$. We split our argument into two cases $m\leq n$ and $m>n$. \\

\subsubsection{Case I:}
    Suppose $m\leq n$. 
        \[\Hom(R_m,R_n) = \left\{ \left[
                    \begin{array}{ccccc}
                        a_1 & a_2      & \cdots & a_{m}\\
                            & \ddots   & \ddots & \vdots \\
                            &          & \ddots & a_2    \\
                            &          &  & a_1 \\
                            &          &        & \\
                            &          &        & 
                    \end{array}
                    \right] \in k^{n \times m} \right\}
        \]
    Let $R_m = \langle f_1,\ldots,f_m \rangle$ and $R_n = \langle h_1,\ldots,h_n \rangle$. Then, we can take $\{\phi_i\}_{i=1}^m$ to be a basis of $\Hom(R_m,R_n)$, where
        \[\phi_i(f_j) = \begin{cases}
                            0, & j=1,\ldots,i-1\\
                            h_{j-i+1}, & j=i,\ldots,n
                        \end{cases}\]
    Note that $\ev[R_m][R_n]$ is not surjective because no $\phi_i$ maps any $f_j$ to $h_k$ for $k=m+1,\ldots,n$. But, we can see that $\phi_1$ embeds a copy of $R_m$ in $R_n$ and so $\ev[R_m][R_n]$ has the largest possible image, $R_m$. So, we have 
        \[\coker \ev[R_m][R_n] = \frac{R_n}{\im \ev[R_m][R_n]} \cong \frac{R_n}{R_m} \cong R_{m-n}.\]
    Let $K = \ker \ev[R_m][R_n]$. We have the follow exact sequence.
        \begin{nalign}
                    &0 \to K \to R_m \otimes \Hom(R_m,R_n) \to R_n \to \coker \ev[R_m][R_n] \to 0\\
        \Rightarrow \, \, &0 \to K \to R_m \otimes \Hom(R_m,R_n) \to R_n \to R_{m-n} \to 0
        \end{nalign}
    which implies
        \begin{nalign} 
            \dim \ker \ev[R_m][R_n] &= \dim R_m \otimes \Hom(R_m,R_n) + \dim R_n - \dim R_{m-n} \\
                                    &= m^2+(n-m)-n\\
                                    &= m^2-m
        \end{nalign}
    We will proceed by defining subrepresentations of $K$ and using Lemma~\ref{lem:recognition} to show that the kernel $K$ is a direct sum of these subrepresentations. For $i=2,\ldots,m$, let
        \begin{nalign}{3} 
            \mathcal{B}_i = &\{f_j \otimes \phi_i \, &&| \, j=1,\ldots,i-1 \} \, \cup\\
                            &\{f_j \otimes \phi_i - f_{j-i+1}\otimes\phi_1 \, &&| \,j=i,\ldots,m\}.
        \end{nalign}
    Define $S_i$ to be the subrepresentation generated by the set $\mathcal{B}_i$ with map $\beta_i$ given by $\beta_{R_m}$. First, we show that $S_i \subseteq K$ for each $i$.

    By definition, we have \[f_j \otimes \phi_i \mapsto \phi_i(f_j) = 0\] for $j=1,\ldots,i-1$. For $j=i,\ldots,m$, 
        \begin{nalign}
            f_j \otimes \phi_i - f_{j-i+1} \otimes \phi_1 &\mapsto \phi_i(f_j) - \phi_1(f_{j-i+1})\\
                                                          &= h_{j-i+1}-h_{j-i+1}\\
                                                          &= 0
        \end{nalign}
    Next, we will show that the set $\mathcal{B} = \cup_{i=2}^m \mathcal{B}_i$ of all elements generating some $S_i$ is linearly independent. Suppose towards contradiction that 
        \begin{nalign}
            0 &= \sum_{i=2}^m \sum_{j=1}^{i-1} r_{ij}(f_j \otimes \phi_i)+ \sum_{i=2}^m \sum_{j=i}^m s_{ij}(f_j \otimes \phi_i - f_{j-i+1} \otimes \phi_1)\\
              &= \sum_{i=2}^m \sum_{j=1}^{i-1} r_{ij}f_j \otimes \phi_i+ \sum_{i=2}^m \sum_{j=i}^m (s_{ij}f_j \otimes \phi_i - s_{ij}f_{j-i+1} \otimes \phi_1)
        \end{nalign}
    Because each of the morphisms $\phi_i$ is a basis element and thus is nonzero and not a linear combination of the others, the terms involving the same $\phi_i$ for a fixed $i$ must cancel out. So, if we consider $i>1$, we must have that 
        \begin{nalign}
            0 &= \sum_{j=1}^{i-1} r_{ij}f_j\otimes \phi_i +  \sum_{j=i}^m s_{ij}f_{j} \otimes \phi_i\\
            \Rightarrow 0 &= \sum_{j=1}^m s_{ij}f_{j}\otimes \phi_i
        \end{nalign}
    where $s_{ij} = r_{ij}$ for $j=1,\ldots,i-1$. Now, because the sets $\{f_j\}_{j=1}^m, \{\phi_i\}_{i=1}^m$ are bases for $R_m, \Hom(R_m,R_n)$, the set of elements $f_j \otimes \phi_i$ forms a basis for $R_m \otimes \Hom(R_m,R_n)$. So, the set of elements $\{f_j \otimes \phi_i\}_{j=1}^m$ for a fixed $i$ are linearly independent from one another. Therefore, $s_{ij}=0$ for $j=i,\ldots,m$. The same is true for all $i$, so $s_{ij}=0$ for all $i, j$ and we have shown that $\mathcal{B}$ is a linearly independent set. 

    Now we will show that $\sum_{i=2}^m \dim S_i = \dim K$. 
        \begin{alignat}{4}
            \dim S_i = |\mathcal{B}_i| & = |\{&&f_j\otimes\phi_i, \quad &&(j=1,\ldots,i-1);\\
                                       &      &&f_j\otimes\phi_i - f_{j-i+1}\otimes\phi_1, \quad &&(j=i,\ldots,m)\}| \nonumber\\
                                       & = m &&  && \nonumber
        \end{alignat}
    so $\sum_{i=2}^m \dim S_i = \sum_{i=2}^m m = (m-1)m = m^2-m = \dim K$, as desired. By Lemma~\ref{lem:recognition}, we have that $K = \oplus_{i=2}^m S_i$. It remains to show that each $S_i \cong R_m$. At the start of the proof, we identified $R_m(x)$ and $R_m(y)$, so we can trace this identification backwards to conclude that there is a map from $S_i$ to itself, which corresponds to a map $\alpha_i = \mathbbm{1}_m$ coming from $\alpha_{R_m}=\mathbbm{1}$. So, we only need to show that $\beta_i = J_m(0)$. In order to show this, we will apply the map $\beta_i$ to the basis elements of $S_i$.
        \[\beta_i(f_j\otimes\phi_i) = \beta_{R_m}(f_j)\otimes\phi_i = \begin{cases}
                0, & j=1\\
                f_{j-1}\otimes\phi_i, & j=2,\ldots,i-1
            \end{cases}\]
        \begin{nalign} 
            \beta_i(f_j\otimes\phi_i-f_{j-i+1}\otimes\phi_1) &= \beta_{R_m}(f_j)\otimes\phi_i - \beta_{R_m}(f_{j-i+1}) \otimes\phi_1 \\ 
                    &= \begin{cases}
                            f_{j-1}\otimes\phi_i, & j=i\\
                            f_{j-1}\otimes\phi_i-f_{j-i}\otimes\phi_1, & j=i+1,\ldots,m
                        \end{cases}
        \end{nalign}
    Therefore, $\beta_i = J_m(0)$, as desired. So, we can conclude that $K \cong (R_m)^{m-1}$. This structure is not dependent on the choice of $\lambda$, so in general, we have that 
        \[\ker \ev[R_m(\lambda)][R_n(\lambda)] \cong (R_m(\lambda))^{m-1}.\]
    
\subsubsection{Case II:}
    Suppose $m >n$. Consider the evaluation morphism 
        \[\ev[R_m][R_n]:R_m\otimes\Hom(R_m,R_n).\] 
    Because $\phi \in \Hom(R_m,R_n)$ has the form $\phi=(A,A)$ and each indecomposable regular representation has an identity map at the arrow $\alpha$, we identity these representations at the vertices $x$ and $y$, and write 
        \[R_m:R_m(y) \xto{\beta_{R_m}} R_m(y).\] 
    Let $R_m = \langle e_1, \ldots, e_{m-n}, f_1,\ldots,f_n\rangle$ and $R_n = \langle h_1,\ldots,h_n\rangle$. If $m>n$, we have  
        \[\Hom(R_m, R_n) = \left\{ \left[
                        \begin{array}{ccccccc}
                          & & a_1 & a_2      & \cdots & a_{n}\\
                         &  &     & \ddots   & \ddots & \vdots \\
                       &  &  &  & \ddots & a_2  \\
                         & &  &  &        & a_1    
                        \end{array}
                        \right] \in k^{n\times m} \right\}
        \]
    with $\dim \Hom(R_m,R_n) = n$. Notice that these matrices have $m-n$ columns of all zeros on the left and we label these elements of $R_m$ as $e_1, \ldots, e_{m-n}$ to highlight this fact. We can define a basis of the $\Hom$-space between $R_m$ and $R_n$ according to the structure above to be given by morphisms $\phi_i$ such that 
        \begin{nalign}
            \phi_i(e_j) &= 0, \qquad \qquad j=1,\ldots,m-n\\
            \vspace{5pt} \\
            \phi_i(f_j) &= \begin{cases}
                                0, & j=1,\ldots,i-1\\
                                h_{j-i+1}, & j=i,\ldots,n
                            \end{cases} 
        \end{nalign}
    for $i=1,\ldots,n$. The morphism $\phi_i$ is chosen so that it corresponds to a matrix $A$ in the description above with $a_i=1$ and $a_k=0$ for $k\neq 1$. The evaluation morphism $\ev[R_m][R_n]$ is surjective because for any $h_j \in R_n$, we can write \[\ev[R_m][R_n](f_{j} \otimes \phi_1) = \phi_1(f_{j}) = h_{j-1+1}=h_j.\] We now calculate the dimension of the kernel:
        \begin{nalign}
            \dim \ker \ev[R_m][R_n] &= \dim R_m \otimes \Hom(R_m,R_n) - \dim R_n\\
                                    &= mn - n
        \end{nalign}
    Our aim is now to describe a set of $mn-n$ basis elements of $K = \ker \ev[R_m][R_n]$. Define
        \[\mathcal{B}_1 := \{ e_{j} \otimes \phi_1 \, | \,  j=1, \ldots,m-n \}\]
    For $i=2,\ldots,n$ define
        \begin{alignat}{2}
            \mathcal{B}_i := &\{e_{j} \otimes \phi_i - e_{j-i+1} \otimes \phi_1 \, &&| \, j = 1,\ldots, m-n\} \, \cup \\
                             &\{f_{j} \otimes \phi_i - e_{m-n+j-i+1} \otimes \phi_1 \, &&| \,j = 1,\ldots, i-1\}\, \cup \nonumber \\
                             &\{f_{j} \otimes \phi_i - f_{j-i+1} \otimes \phi_1 \, &&| \,j = i, \ldots, n\} \nonumber
        \end{alignat}
    Note that the elements $f_k,e_k$ are zero when $k<1$. Let $S_i$ to be the representation generated by $\mathcal{B}_i$ with the map $\beta_i: S_i \to S_i$ given by the action of $\beta_{R_m}$. Using Lemma~\ref{lem:recognition}, we will show that $K = \oplus_{i=1}^n S_i$.
    
    First we prove that each $\mathcal{B}_i$ is contained in the kernel $K$. For the elements in $\mathcal{B}_1$, we have 
        \begin{nalign} e_j \otimes \phi_1 &\mapsto \phi_1(e_j) \\ &= 0\end{nalign}
    by definition for $j=1,\ldots,m-n$. For the first type of basis element in $\mathcal{B}_i$, $i=2,\ldots,n$, we have
        \begin{nalign} e_j \otimes \phi_i - e_{j-i+1} \otimes \phi_1 &\mapsto \phi_i(e_j) - \phi_1(e_{j-i+1}) \\ &= 0-0 \\ & = 0\end{nalign}
    for $j=1,\ldots,i-1$. For the second type, we have
        \begin{nalign} f_j \otimes \phi_i - e_{m-n+j-i+1} \otimes \phi_1 &\mapsto \phi_i(f_j) - \phi_1(e_{m-n+j-i+1}) \\ &= h_{j-i+1} - 0 \\ &= 0.\end{nalign}
    Since $j=1,\ldots, i-1$, $j-i+1 < 1$ and so $ h_{j-i+1} = 0$. Finally, for the third type, we have $j=i,\ldots,n$, so
        \begin{nalign} f_j \otimes \phi_i - f_{j-i+1} \otimes \phi_1 &\mapsto \phi_i(f_j) - \phi_1(f_{j-i+1}) \\ &= h_{j-i+1} - h_{j-i+1}\\  &= 0.\end{nalign}
    Therefore, $S_i \subseteq K$ for $i=1,\ldots, n$. Let $\mathcal{B} = \cup_{i=1}^n \mathcal{B}_i$ be the set of all elements generating some $S_i$, for $i=1,\ldots,n$. Now suppose we have a linear combination
        \begin{alignat}{2}
            0 &= \sum_{j=1}^{m-n} r_j (e_j \otimes \phi_1) &&+ \sum_{i=2}^n \sum_{j=1}^{m-n} s_{ij} (e_j \otimes \phi_i - e_{j-i+1}\otimes \phi_1) \\ 
                    &&&+ \sum_{i=2}^n \sum_{j=1}^{i-1} t_{ij}(f_j \otimes \phi_i - e_{m-n+j-i+1}\otimes\phi_1) \nonumber\\ 
                    &&&+ \sum_{i=2}^n \sum_{j=i}^{n} u_{ij}(f_j \otimes \phi_i - f_{j-i+1}\otimes\phi_1)\nonumber
        \end{alignat}
    Note that the coefficients of the linear combination can be applied so that they only appear in the first argument of the tensor. Also, because the elements $\{\phi_i\}$ for a basis of $\Hom(R_m,R_n)$, none of them are zero. Therefore, in order to sum to zero, the terms involving $\phi_i$ for each fixed $i=1,\ldots,n$ must cancel each other out. For a fixed $i > 1$, we have
        \begin{nalign}
            0  &= \sum_{j=1}^{m-n} s_{ij}e_j \otimes \phi_i +  \sum_{j=1}^{i-1} t_{ij}f_j \otimes \phi_i + \sum_{j=i}^{n} u_{ij}f_j \otimes \phi_i\\
                &= \left( \sum_{j=1}^{m-n} s_{ij}e_j +  \sum_{j=1}^{i-1} t_{ij}f_j + \sum_{j=i}^{n} u_{ij}f_j \right) \otimes \phi_i 
        \end{nalign}
    Because the $\phi_i$ is a basis element of $\Hom(R_m,R_n)$, it is not zero, so we must have that the sum in the first argument of the tensor is zero: 
        \[0 = \sum_{j=1}^{m-n} s_{ij}e_j +  \sum_{j=1}^{i-1} t_{ij}f_j + \sum_{j=i}^{n} u_{ij}f_j \]
    This is a linear combination of elements which form a basis of $R_m$ and so we conclude that $s_{ij} = t_{ij} = u_{ij} = 0$ for each $j$. The argument is independent of $i>1$, so we conclude that $s_{ij} = t_{ij} = u_{ij} = 0$ for each $i$ as well as $j$. Going back to the original sum, we are only left with the first term 
        \[0 = \sum_{j=1}^{m-n} r_{j} e_j \otimes \phi_1.\]
    This implies that $\sum_{j=1}^{m-n}r_j e_j = 0$ and because the $\{e_j\}_{j=1}^{m-n}$ is a subset of a basis, it must be that $r_j=0$ for all $j$. Therefore, the set $\mathcal{B}$ is linearly independent.

    Next, we count the dimension of each $S_i$:
        \begin{alignat}{3}
            \dim S_1  = |\mathcal{B}_1| & = \, &&|\{e_j\otimes \phi_1 \, | \, j=1,\ldots, m-n\}| \\
                     & = &&m-n  \nonumber\\
            \nonumber\\
            \dim S_{i > 1} = |\mathcal{B}_i| & = \, &&|\{e_{j} \otimes \phi_i - e_{j-i+1} \otimes \phi_1 \quad \, | \, j = 1,\ldots, m-n\} \\
                     & &&+ \{f_{j} \otimes \phi_i - e_{m-n+j-i+1} \otimes \phi_1 \, | \, j = 1,\ldots, i-1\} \nonumber\\
                     & &&+ \{f_{j} \otimes \phi_i - f_{j-i+1} \otimes \phi_1 \, | \, j = i, \ldots, n\}|  \nonumber\\
                     & = \, &&m-n+n  \nonumber\\      
                     & = \, &&m \nonumber
        \end{alignat}
    Because there are $n-1$ subrepresentations $S_i$ with $i \neq 1$, the dimension of $S$ is \[\sum_i \dim S_i = m-n + (n-1)m = mn-n,\] which matches the dimension of the kernel $K$. By Lemma~\ref{lem:recognition}, we conclude that $K = \oplus_{i=1}^n S_i$

    Finally, we must show that $S_1 \cong R_{m-n}$ and that $S_i \cong R_m$ for $i=2,\ldots,n$. We will do so by applying $\beta_i$ to each subrepresentation $S_i$. Recall that we identified $R_m(x)$ and $R_m(y)$ using that $\alpha_{R_m}= \mathbbm{1}_m$. Thus, we have $S_i(x) = S_i(y)$, as well, and $\alpha_1 = \mathbbm{1}_{m-n}$ and $\alpha_i = \mathbbm{1}_{m}$ for $i=2,\ldots,n$.  
    
    For $i=1$, we have 
        \begin{nalign}
            \beta_1(e_j \otimes \phi_1) = \beta_{R_m}(e_j) \otimes \phi_1 = \begin{cases} 0, & j=1 \\ e_{j-1} \otimes \phi_1, & j=2,\ldots,m-n. \end{cases}
        \end{nalign}
    So, $\beta_1 = J_{m-n}(0)$ which means that $S_1 \cong R_{m-n}$. For $i>1$,
        \begin{nalign}
            \beta_i(e_j\otimes \phi_i - e_{j-i+1} \otimes \phi_1) &= \beta_{R_m}(e_j)\otimes\phi_i - \beta_{R_m}(e_{j-i+1}) \otimes \phi_1\\
                &= \begin{cases}0, & j=1\\ e_{j-1}\otimes\phi_i - e_{j-i}\otimes\phi_1, & j=2,\ldots, m-n\end{cases}
        \end{nalign}
    Note that for $j=1$, we have $e_0 \otimes \phi_i - e_{1-i}\otimes\phi_1 = 0-0 = 0$. Next, we have   
        \begin{nalign}
            \beta_i(f_j\otimes \phi_i - e_{m-n+j-i+1} \otimes \phi_1) &= \beta_{R_m}(f_j)\otimes\phi_i - \beta_{R_m}(e_{m-n+j-i+1}) \otimes \phi_1\\
                &= \begin{cases}e_{m-n} \otimes \phi_i - e_{m-n-i+1}\otimes\phi_1, & j=1\\ f_{j-1}\otimes\phi_i - e_{m-n+j-i}\otimes\phi_1, & j=2,\ldots, i-1.\end{cases}
        \end{nalign}
    In the case that $j=1$, we use that $\beta_{R_m}{f_1}=e_{m-n}$. Finally, 
        \begin{nalign}
            \beta_i(f_j\otimes \phi_i - f_{j-i+1} \otimes \phi_1) &= \beta_{R_m}(f_j)\otimes\phi_i - \beta_{R_m}(f_{j-i+1}) \otimes \phi_1\\
                &= \begin{cases}f_{i-1} \otimes \phi_i - e_{m-n}\otimes\phi_1, & j=i\\ f_{j-1}\otimes\phi_i - f_{j-i}\otimes\phi_1, & j=i+1,\ldots,n.\end{cases}
        \end{nalign}
    Combining these results, we conclude that $\beta_i = J_m(0)$ for $i=2,\ldots,n$. Hence, $S_i \cong R_m$. We now conclude that $K \cong R_{m-n} \oplus (R_m)^{n-1}$. As noted at the beginning of the proof, this structure is independent of $\lambda$, so more generally, we have 
        \[\ker \ev[R_m(\lambda)][R_n(\lambda)] \cong R_{m-n}(\lambda) \oplus (R_m(\lambda))^{n-1}.\] 
    Combined with the result from Case I, we have obtained the desired result.
\end{proof}

\subsection{From Regular to Preinjective}

\begin{customthm}{\ref{thm:ev(Rm,In)}}
    Given a regular indecomposable representation $R_m(\lambda)$, $m \geq 1, \lambda \in k \cup \{\infty\}$, and a preinjective indecomposable representation $I_n, n \geq 0$, of the Kronecker quiver $\mathsf{K}$, we have that 
        \begin{equation}\ker \ev[R_m(\lambda)][I_n] \cong P_{m-n-1} \oplus (R_m(\lambda))^{m-1}, \quad m \geq n+1 \text{ or } m>1.\end{equation}
    All other kernels are equal to 0.
\end{customthm}

\begin{proof}
    As with the previous lemma, our first step is understanding the evaluation morphism. We will proceed in two cases, $\ev[R_m][I_0]$ and $\ev[R_m][I_1]$, where $R_m \coloneqq R_m(0)$. We note that by \cite{ElementsVol2} $\Hom(R_m(\lambda),N) \cong \Hom(R_m(\mu),N)$ for $\lambda, \mu \in k \cup \{\infty\}$, so we may use the simplest case of $\lambda = 0$. \\

\subsubsection{Case I:}
    Consider the morphism \[\ev[R_m][I_0]: R_m \otimes \Hom(R_m,I_0) \to I_0.\] By Proposition~\ref{prop:hom_isos}, we know that $\Hom(R_m,I_0) \cong \Hom(R_m, I(x)) \cong (R_m(x))^*$, the dual of the vector space at the vertex $x$ for the representation $R_m$. Thus, the evaluation morphism can be rewritten to be
        \[(\langle e_1, \ldots, e_m\rangle \toto \langle f_1, \ldots, f_m \rangle) \otimes (R_m(x))^* \to (k \toto 0)\]
    where we have let $R_m = \langle e_1, \ldots, e_m\rangle \toto \langle f_1, \ldots, f_m \rangle$ and $I_m:k \toto 0$. Note that $R_m(x))^*$ has dual basis $\langle e_1^*, \ldots, e_m^* \rangle$. Distributing the tensor yields:
        \[\langle e_1, \ldots, e_m\rangle \otimes (R_m(x))^* \toto \langle f_1, \ldots, f_m \rangle \otimes (R_m(x))^* \to (k \toto 0)\]
    We know that an morphism in $\Hom(R_m,I_0)$ is a pair of maps $(A,B)$ such that the following diagram commutes.
\begin{equation}\begin{tikzcd}[ampersand replacement=\&,column sep=large]
	{k^m} \& {} \& k \\
	\\
	{k^m} \& {} \& 0
	\arrow["{J_m(0)}", shift left=3, from=1-1, to=3-1]
	\arrow["{\mathbbm{1}_m}"', shift right=2, from=1-1, to=3-1]
	\arrow["{A = [a_1 \, \ldots \, a_m]}", shorten <=9pt, shorten >=9pt, from=1-1, to=1-3]
	\arrow["{B = 0}"', shorten <=9pt, shorten >=9pt, from=3-1, to=3-3]
	\arrow["0", shift left=2, from=1-3, to=3-3]
	\arrow["0"', shift right=2, from=1-3, to=3-3]
\end{tikzcd}\end{equation}
    Because both maps $\alpha_{I_0}, \beta_{I_0} = 0$, there are not commutativity restraints to ensure compatibility with the maps of $I_0$. The map $B$ must be $0$ and the map $A$ is given by the evaluation of the dual basis on the basis of $R_m$, acting as $e_j \otimes \phi_i \mapsto \phi_i(e_j) = \delta_{j}^i $, where $\phi_i = e_i^*$. Let $K = \ker \ev[R_m][I_0]$.
    
    The evaluation morphism $\ev[R_m][I_0]$ is surjective, so we may calculate the dimension of the kernel as follows:
            \begin{nalign}
                \ddim K &= \ddim R_m \otimes \Hom(R_m,I_0) - \ddim I_0 \\
                        &= \ddim R_m \otimes R_m(x)^* - \ddim I_0 \\
                        &= \ddim R_m^m - \ddim I_0 \\
                        &= m(m,m) - (1,0)\\
                        &= (m^2-1,m^2)
            \end{nalign}
    Our goal is to prove the decomposition of $K$ into indecomposable summands. We will use Lemma~\ref{lem:recognition}. First, define the following collections of elements of $K$.
        \begin{nalign}
            \mathcal{B}_1(x) &= \{e_j \otimes \phi_1 \, | \, j=2,\ldots,m\} \\
            \mathcal{B}_1(y) &= \{f_j\otimes\phi_1 \, | \, j=1,\ldots,m\}
        \end{nalign} and
        \begin{nalign}
            \mathcal{B}_i(x) &= \{e_j \otimes \phi_i - e_{j-i+1} \otimes \phi_1 \, | \, j=1,\ldots,m\}\\
            \mathcal{B}_i(y) &= \{f_j \otimes \phi_i - f_{j-i+1} \otimes \phi_1\, | \, j=1,\ldots,m\}
        \end{nalign}
    for $i=2,\ldots, m$. Note that the terms with $e_k$ or $f_k$ with $k<1$ are equal to 0. Let $S_i$ be the subrepresentation of $K$ with $S_i(x)$ is generated by $\mathcal{B}_i(x)$ and similarly at the vertex $y$. These representations each have maps $\alpha_i, \beta_i$ corresponding to the arrows $\alpha, \beta$ of the Kronecker quiver $\mathsf{K}$, respectively.
    
    First, we will show that $S_i \subseteq K$ for each $i$. First, we note that all elements involving $f_j$ are in the kernel because the map $B$ is zero. Therefore, we focus on showing that the elements at the vertex $x$ (involving $e_j$) are in the kernel:
        \begin{nalign}
            e_j \otimes \phi_1 &\mapsto \phi_1(e_j)\\
                               &= e_1^*(e_j)\\
                               &= \delta_j^i = 0; \quad 
        \end{nalign}
    for $j = 2,\ldots,m$. Similarly, we have 
        \begin{nalign}    
            e_j \otimes \phi_i - e_{j-i+1} \otimes \phi 1 &\mapsto \phi_i(e_j)-\phi_1(e_{j-i+1})\\
                                                                &= e_i^*(e_j) - e_1^*(e_{j-i+1})\\
                                                                &= \delta_j^i - \delta_{j-i+1}^1
    \end{nalign}
    for $j=1,\ldots,m$. We notice that $i=j \Leftrightarrow j-i+1=1$, so both terms are nonzero in the same cases. We then get \[\delta_j^i - \delta_{j-i+1}^1=\delta_1^1 - \delta_1^1 = 0.\] In all other cases, both terms are zero, so we get 0 overall. We conclude that $S_i \subseteq K$ for $i=1,\ldots,m$. 
    
    Next we will show that $\mathcal{B}= \cup_i \mathcal{B}_i$ is a linearly independent set. First, we prove the linear independence of the elements in $\mathcal{B}(x) = \cup_i \mathcal{B}_i(x)$. Suppose that we have a linear combination 
        \begin{nalign}
            0 &= \sum_{j=2}^m r_j (e_j \otimes \phi_1) + \sum_{i=2}^m \sum_{j=1}^{m} s_{ij} (e_j \otimes \phi_i - e_{j-i+1} \otimes \phi_1) \\
              &= \sum_{j=2}^m r_j e_j \otimes \phi_1 + \sum_{i=2}^m\sum_{j=1}^{m} s_{ij} e_j \otimes \phi_i - s_{ij} e_{j-i+1} \otimes \phi_1
        \end{nalign}
    Because the $\phi_i$ in the second argument form a basis, they are linearly independent and thus no $\phi_i$ can be written as a linear combination of the others. This means that for a fixed $i$, the terms coefficients of $\phi_i$ must cancel. Thus, if we restrict our view to the terms involving $\phi_i$ for $i>1$, we must have
        \begin{nalign}
            0 &= \sum_{j=1}^{m} s_{ij} (e_j \otimes \phi_i)\\
              &= \sum_{j=1}^{m} s_{ij} e_j \otimes \phi_i
        \end{nalign}
    Since $\phi_i$ is not 0, we must have 
        \[0 =\sum_{j=1}^{m} s_{ij} e_j\]
    But, this is a linear combination of the elements $\{e_j\}_j$ which form a basis for $R_m(x)$. Therefore, it must be that $s_{ij} = 0$ for $j=1,\ldots,m$. Because the argument is the same for each $i>1$, we conclude that $s_{ij}=0$ for $i=2,\ldots,m$ and $j=1,\ldots,m$. The original linear combination thus reduces to 
        \[0 = \sum_{j=2}^m r_j e_j \otimes \phi_1\]
    In the same way as before, we conclude that 
        \[0 = \sum_{j=2}^m r_j e_j\]
    because $\phi_1 \neq 0$. Then, because $\{e_j\}_j$ is a basis, it must be that $r_j=0$ for $j=2,\ldots,m$. Therefore, all of the coefficients in the linear combination must be zero and the elements of $\mathcal{B}(x)$ are linearly independent. A similar argument works for $\mathcal{B}(y)$, using that the set $\{f_j\}$ is a basis of $R_m(y)$.
    
    Now we will show that $\sum_{i=1}^m \ddim S_i = \ddim K$. 
        \begin{nalign}
            \sum_{i=1}^m \ddim S_i &= \ddim S_1 + \sum_{i=2}^m \ddim S_i \\
                                   &= |\mathcal{B}_1| + \sum_{i=2}^m |\mathcal{B}_i| \\
                                   &= (m-1,m) + (m-1)(m,m)\\
                                   &= (m-1,m) + (m^2-m, m^2-m)\\
                                   &= (m^2-m,m^2)\\
                                   &= \ddim K
        \end{nalign}
    By Lemma~\ref{lem:recognition}, we can conclude that $\sum_{i=1}^m S_i =\otimes_{i=1}^m S_i = K$. It remains to show that $S_1 \cong P_{m-1}$ and $S_i \cong R_m$ for $i=2,\ldots,m-1$. Recall that \[S_1: \langle e_j \otimes \phi_1 | j=2,\ldots,m \rangle \toto \langle f_j\otimes\phi_1 | j=1,\ldots,m\rangle\] with maps $\alpha_1, \beta_1$ given by the action of $\alpha_{R_m}, \beta_{R_m}$. So, for $j=2,\ldots,m$, we have
        \begin{nalign}
            \alpha_1(e_j\otimes\phi_1) &= \alpha_{R_m}(e_j)\otimes\phi_1 = f_j \otimes \phi_1\\
            \beta_1(e_j\otimes\phi_1) &= \beta_{R_m}(e_j)\otimes\phi_1 = f_{j-1} \otimes \phi_1
        \end{nalign}
    which gives us that $\alpha_i = \begin{bsmallmatrix} 0 \\ \mathbbm{1}_{m-1} \end{bsmallmatrix}$ because no element $e_j \otimes \phi_1$ is sent to $f_1 \otimes \phi_1$ and this results in a row of all zeros at the top of the matrix, while all other $m-1$ elements $e_j$ are sent to the corresponding $f_j$. We also get that $\beta_i = \begin{bsmallmatrix} \mathbbm{1}_{m-1} \\ 0\end{bsmallmatrix}$ because no element is sent to $f_m \otimes \phi_1$. If we reverse the order of the basis elements for $S_1$, the matrices will match $\alpha_{P_{m-1}}, \beta_{P_{m-1}}$, so we easily see that $S_1 \cong P_{m-1}$.
    
    We repeat the process to show that $S_i \cong R_m$, for $i=2,\ldots,m$.
        \begin{nalign}
            \alpha_i(e_j\otimes\phi_1 - e_{j-i+1} \otimes \phi_1) &= \alpha_{R_m}(e_j)\otimes\phi_i - \alpha_{R_m}(e_{j-i+1}) \otimes \phi_1 \\
                                &= f_j \otimes \phi_i -f_{j-i+1} \otimes \phi_1
        \end{nalign}
    for $j=1,\ldots,m$. We also have
        \begin{nalign}    
            \beta_i(e_j\otimes\phi_1 - e_{j-i+1} \otimes \phi_1) &= \beta_{R_m}(e_j)\otimes\phi_i - \beta_{R_m}(e_{j-i+1}) \otimes\phi_1 \\
                                &= f_{j-1} \otimes \phi_i - f_{j-i}\otimes\phi_1
        \end{nalign}   
    for $j=2,\ldots,m$. We have $\beta_{R_m}(e_1) = 0$ by definition and the element $f_{1-i}=0$ because for all $i \geq 2, 1-i<1$. Therefore, the $\beta_i$ sends the first term to 0 and we can conclude that $\alpha_i = \mathbbm{1}_m, \beta_i = J_m(0)$. So, $S_i \cong R_m$ for $i=2,\ldots,m$. Hence, $K \cong P_{m-1} \oplus (R_m)^{m-1}$, as desired.
    
    Now, we can use Corollary~\ref{cor:ker_ev_t} to extend this formula for the kernel to general $I_n$ with $\tau^{-k} I_n \cong I_0$ ($n=2k$):
        \begin{nalign}
            \ev[R_m][I_n] &\cong \tau^{k} \ev[\tau^{-k} R_m][\tau^{-k} I_n]\\
                          &\cong \tau^k \ev[R_m][I_0]\\
                          &\cong \tau^k(P_{m-1} \oplus (R_m)^{m-1}\\\
                          &\cong P_{m-2k-1} \oplus (R_m)^{m-1}\\
                          &\cong P_{m-n-1} \oplus (R_m)^{m-1}
        \end{nalign}
    which agrees with the desired formula. Note that when $m < n+1$, we lose the summand $P_{m-n-1}$.\\

\subsubsection{Case II:}
    Consider the morphism \[\ev[R_m][I_1]: R_m \otimes \Hom(R_m,I_1) \to I_1.\] As in the previous case, our first step is to understand the structure of $\Hom(R_m,I_1)$. We will calculate this space directly. First, we know that a map $\phi \in \Hom(R_m,I_1)$ is a pair of maps $(A,B)$ which commute with the maps $\alpha_{R_m}, \beta_{R_m}, \alpha_{I_1}, \beta_{I_1}$ according to the diagram below. 
\begin{equation}\begin{tikzcd}[ampersand replacement=\&,column sep=large]
	{k^m} \& {} \& {k^2} \\
	\\
	{k^m} \& {}\& k
	\arrow["{\mathbbm{1}_m}"', shift right=2, from=1-1, to=3-1]
	\arrow["{J_m(0)}", shift left=2, from=1-1, to=3-1]
	\arrow["{\begin{bsmallmatrix} 1 & 0 \end{bsmallmatrix}}"', shift right=2, from=1-3, to=3-3]
	\arrow["{\begin{bsmallmatrix} 0 & 1 \end{bsmallmatrix}}", shift left=2, from=1-3, to=3-3]
	\arrow["{A = \begin{bsmallmatrix} a_{11} & \ldots & a_{1m}\\ a_{21} & \ldots & a_{2m} \end{bsmallmatrix}}", from=1-1, to=1-3]
	\arrow["{B = \begin{bsmallmatrix} b_1 & \ldots & b_m \end{bsmallmatrix}}"', from=3-1, to=3-3]
\end{tikzcd}\end{equation}
    So, we must have 
        \begin{nalign}
            \begin{bsmallmatrix} 1 & 0 \end{bsmallmatrix} \, A = B \, \mathbbm{1}_m &\Rightarrow \begin{bsmallmatrix}a_{11}  & \ldots &  a_{1m}\end{bsmallmatrix} = \begin{bsmallmatrix}b_1 & \ldots & b_m\end{bsmallmatrix} \\
            \begin{bsmallmatrix}0 & 1\end{bsmallmatrix} \, A = B \, J_m(0) &\Rightarrow \begin{bsmallmatrix}a_{21} & \ldots & a_{2m}\end{bsmallmatrix} = \begin{bsmallmatrix}0 & b_1 & \ldots & b_{m-1}\end{bsmallmatrix}
        \end{nalign}
    Therefore, the maps $A$ and $B$ can be written as
        \[A = \begin{bmatrix} a_1 & a_2 & \ldots & a_m \\ 0 & a_1 & \ldots & a_{m-1}\end{bmatrix} \quad \text{ and } \quad 
          B = \begin{bmatrix} a_1 & \ldots & a_m \end{bmatrix}\]
    where $a_i \in k$. Define $\phi_i = (A_i,B_i)$ where $A_i, B_i$ denote the matrices $A,B$ with $a_i = 1$ and $a_k=0$ for $k \neq i$. It is clear that the set $\{\phi_i\}_{i=1}^m$ is then a basis of $\Hom(R_m,I_1)$. 
    
    Let \(R_m = \langle e_1,\ldots,e_m \rangle \toto \langle f_1,\ldots,f_m \rangle \quad \text{ and } \quad I_1 = \langle g_1,g_2 \rangle \toto \langle h_1 \rangle.\) Then, the map $\phi_i$ for $i=1,\ldots,m$ is given on these bases by
        \begin{nalign}
            \phi_i(e_j) &= \begin{cases}g_1, & j=i\\ g_2, & j=i+1\\ 0, & j \neq i, i+1 \end{cases}\\
            &\\
            \phi_i(f_j) &= \begin{cases}h_1, & j=i\\ 0, & j\neq i \end{cases} 
        \end{nalign}
    for $j=1,\ldots,m$. Let $K = \ker \ev[R_m][I_1]$. Note that $\ev[R_m][I_1]$ is surjective because we can write
        \begin{nalign}
            g_1 &= \phi_1(e_1) = \ev[R_m][I_1](e_1 \otimes \phi_1)\\
            g_2 &= \phi_1(e_1) = \ev[R_m][I_1](e_2 \otimes \phi_1)\\
            h_1 &= \phi_1(f_1) = \ev[R_m][I_1](f_1 \otimes \phi_1)
        \end{nalign}
    Therefore, we can easily calculate $\ddim K$:
        \[\ddim K = \ddim R_m \otimes \Hom(R_m,I_1) - \ddim I_1 = m(m,m)-(2,1) = (m^2-2,m^2-1).\]
    Let us define the following collections of elements in $K$.
        \begin{nalign}
            \mathcal{B}_1(x) = \{&e_j \otimes \phi_1 \, | \, j=3,\ldots,m\} \\
            \mathcal{B}_1(y) = \{&f_j\otimes\phi_1 \, | \, j=2\ldots,m\}\\
            & \\
            \mathcal{B}_i(x) = \{&e_j \otimes \phi_i \, | \, j=1,\ldots, i-1\}  \,\cup \,
                               \{e_j \otimes \phi_i - e_{j-i+1} \otimes \phi_1 \, | \,j=i,\ldots, m\} \\
            \mathcal{B}_i(y) = \{&f_j \otimes \phi_i \, | \, j=1,\ldots, i-1\} \, \cup \, \{f_j \otimes \phi_i - f_{j-i+1} \otimes \phi_1 \, | \,j=i,\ldots, m\}
        \end{nalign}
    for $i=2,\ldots,m$. Let $S_i$ be the representation such that $S_i(x)$ is generated by $\mathcal{B}_i(x)$ and similarly at the vertex $y$. These representations each have maps $\alpha_i, \beta_i$ given by the action of $\alpha_{R_m}, \beta_{R_m}$, respectively.
    
    To begin, we show that $S_i \subseteq K$ for each $i$. 
        \begin{nalign}
            e_j \otimes \phi_1 &\mapsto \phi_1(e_j) = 0
        \end{nalign}
    for $j = 3,\ldots,m$, and
        \begin{nalign}
            f_j \otimes \phi_1 &\mapsto \phi_1(f_j) = 0
        \end{nalign}
    for $j = 2,\ldots,m$. So, $S_1 \subseteq K$. For $i=2,\ldots,m$, we have
        \begin{nalign}
            e_j \otimes \phi_i &\mapsto \phi_i(e_j) = 0 \quad \text{ for } j=1,\ldots,i-1\\ \\
            e_j \otimes \phi_i - e_{j-i+1} \otimes \phi 1 &\mapsto \phi_i(e_i)-\phi_1(e_{1})\\
                                                                &= g_1-g_1\\
                                                                &= 0 \quad \text{ for } j=i\\ \\
            e_j \otimes \phi_i - e_{j-i+1} \otimes \phi 1 &\mapsto \phi_i(e_{i+1})-\phi_1(e_{2})\\
                                                                &= g_2-g_2\\ \
                                                                &= 0 \quad \text{ for } j=i+1\\ \\
            e_j \otimes \phi_i - e_{j-i+1} \otimes \phi 1 &\mapsto \phi_i(e_j)-\phi_1(e_{j-i+1})\\
                                                                &= 0-0\\
                                                                &= 0 \quad \text{ for } j=i+2,\ldots,m
        \end{nalign}
    and at vertex $y$, 
        \begin{nalign}
            f_j \otimes \phi_i &\mapsto \phi_i(f_j) = 0 \quad \text{ for } j=1,\ldots,i-1\\ \\
            f_j \otimes \phi_i - f_{j-i+1} \otimes \phi 1 &\mapsto \phi_i(f_i)-\phi_1(f_{1})\\
                                                                &= h_1-h_1\\ 
                                                                &= 0 \quad \text{ for } j=i\\ \\
            f_j \otimes \phi_i - f_{j-i+1} \otimes \phi 1 &\mapsto \phi_i(f_j)-\phi_1(f_{j-i+1})\\
                                                                &= 0-0\\
                                                                &= 0 \quad \text{ for } j=i+1,\ldots,m
        \end{nalign}
    So, $S_i \subseteq K$ for all $i$. Next we will prove the linear independence of the elements in $\mathcal{B}(x) = \cup_i \mathcal{B}_i(x)$. Suppose that we have a linear combination 
        \begin{nalign}
            0 = \sum_{j=3}^m r_j (e_j \otimes \phi_1) &+ \sum_{i=2}^m \sum_{j=1}^{i-1} s_{ij} (e_j \otimes \phi_i) +\sum_{i=2}^m \sum_{j=i}^{m} s_{ij} (e_j \otimes \phi_i - e_{j-i+1} \otimes \phi_1)
        \end{nalign}
    where we note that there is no issue in the labeling of the coefficients $s_{ij}$ as the second and third terms sum over different $j$. As in the previous case, we first restrict our attention to the terms involving $\phi_i$ for a fixed $i>1$ which must sum to 0.
        \begin{nalign}
            0 &= \sum_{j=1}^{i-1} s_{ij} e_j \otimes \phi_i + \sum_{j=i}^m s_{ij} e_{j} \otimes \phi_i\\
              &= \sum_{j=1}^{m} s_{ij} e_j \otimes \phi_i 
        \end{nalign}
    Since $\phi_i$ is not 0, we must have 
        \[0 = \sum_{j=1}^{m} s_{ij} e_j.\]
    This is now a linear combination of the elements $\{e_j\}_j$ which form a basis for $R_m(x)$. Therefore, it must be that $s_{ij} = 0$ for all $j$. They are also $0$ for each $i>1$ as the argument is the same in these cases. So, the original linear combination is reduced to \[0 = \sum_{j=3}^m r_j e_j \otimes \phi_1.\] We have that $r_j=0$ for $j=3,\ldots,m$ because $\{e_j\}_j$ is a basis, so is a linearly independent set. Now all of the coefficients of the linear combination have been shown to equal, so the set $\mathcal{B}(x)$ is linearly independent. A similar argument works for $\mathcal{B}(y)$, using that the set $\{f_j\}_j$ is a basis of $R_m(y)$.
    
    Now we would like to show that $\sum_{i=1}^m \ddim S_i = \ddim K$. 
        \begin{nalign}
            \ddim S_1(x) &= |\mathcal{B}_1(x)| = |\{e_j \otimes\phi_1 \, |\, j=3,\ldots,m\}| = m-2\\
            \nonumber\\
            \ddim S_1(y) &= |\mathcal{B}_1(y)| = |\{f_j \otimes\phi_1 \, |\, j=2,\ldots,m\}| = m-1\nonumber
        \end{nalign}
    which gives us that $\ddim S_1 = (m-2,m-1)$. 
        \begin{alignat}{3}
            \ddim S_i(x) = |\mathcal{B}_i(x)|  & = \,\, &&|\{e_j\otimes\phi_i \,|\, j=1,\ldots,i-1\}|\\
                                               & &&+|\{e_j\otimes\phi_i - e_{j-i+1}\otimes\phi_1 \,|\, j=i,\ldots,m\}|\nonumber\\
                                               & = \,\, && m \nonumber\\
                                              \nonumber \\
            \ddim S_i(y) = |\mathcal{B}_i(y)|  & = \,\, &&|\{f_j\otimes\phi_i \,|\, j=1,\ldots,i-1\}|\\
                                               & &&+|\{f_j\otimes\phi_i - f_{j-i+1}\otimes\phi_1 \,|\, j=i,\ldots,m\}|\nonumber\\
                                               & = \,\, && m \nonumber                                 
        \end{alignat}
    and so $\ddim S_i = (m,m)$ for $i=2,\ldots,n$. Combining these results, we see that
        \begin{nalign}
            \sum_{i=1}^m \ddim S_i &= \ddim S_1 + \sum_{i=2}^m \ddim S_i \\
                                   &= (m-2,m-1) + (m-1)(m,m)\\
                                   &= (m-2,m-1) + (m^2-m, m^2-m)\\
                                   &= (m^2-2,m^2-1)\\
                                   &= \ddim K
        \end{nalign}
    Therefore, by Lemma~\ref{lem:recognition}, we can conclude that $K =\otimes_{i=1}^m S_i$. It remains to show that $S_1 \cong P_{m-2}$ and $S_i \cong R_m$ for $i=2,\ldots,m$. First, applying $\alpha_1, \beta_1$ to the basis elements of $S_1(x)$ yields 
        \begin{nalign}
            \alpha_1(e_j\otimes\phi_1) &= \alpha_{R_m}(e_j)\otimes\phi_1 = f_j \otimes \phi_1\\
            \beta_1(e_j\otimes\phi_1) &= \beta_{R_m}(e_j)\otimes\phi_1 = f_{j-1} \otimes \phi_1
        \end{nalign}
    for $j=3,\ldots,m$, which gives us that $\alpha_1 = \begin{bsmallmatrix} 0_{1\times(m-2)} \\ \mathbbm{1}_{m-2} \end{bsmallmatrix}$ because no element $e_j \otimes \phi_1$ is sent to $f_2 \otimes \phi_1$ and this results in a row of all zeros at the top of the matrix, while all other $m-1$ elements $e_j$ are sent to the corresponding $f_j$. We also get that $\beta_i = \begin{bsmallmatrix} \mathbbm{1}_{m-2} \\ 0_{1\times(m-2)}\end{bsmallmatrix}$ because no element is sent to $f_m \otimes \phi_1$. If we reverse the order of the basis elements for $S_1$, the matrices will match those of $\alpha_{P_{m-2}}, \beta_{P_{m-2}}$, so we easily see that $S_1 \cong P_{m-2}$.
    
    We repeat the process to show that $S_i \cong R_m$, for $i=2,\ldots,m$.
        \begin{nalign}
            \alpha_i(e_j\otimes\phi_i) &= \alpha_{R_m}(e_j)\otimes\phi_i\\
                                       &=  f_j\otimes\phi_i
        \end{nalign}
    for $j=1,\ldots,i-1$. Also,
        \begin{nalign}
            \alpha_i(e_j\otimes\phi_i - e_{j-i+1} \otimes \phi_1) &= \alpha_{R_m}(e_j)\otimes\phi_i - \alpha_{R_m}(e_{j-i+1}) \otimes \phi_1 \\
                                &= f_j \otimes \phi_i -f_{j-i+1} \otimes \phi_1
        \end{nalign}
    for $j=i,\ldots,m$. Together these calculations tell us that the basis elements of $S_i(x)$ are sent to the corresponding elements with the same $j$ at $S_i(y)$. So, $\alpha_i = \mathbbm{1}_m$. We also have
        \begin{nalign}
            \beta_i(e_j\otimes\phi_i) &= \beta_{R_m}(e_j)\otimes\phi_i\\
                                      &=  \begin{cases}0, & j=1\\ f_{j-1}\otimes\phi_i, & j=2,\ldots,i-1 \end{cases}
        \end{nalign}
    and 
        \begin{nalign}
            \beta_i(e_j\otimes\phi_i - e_{j-i+1} \otimes \phi_1) &= \beta_{R_m}(e_j)\otimes\phi_i - \beta_{R_m}(e_{j-i+1}) \otimes\phi_1 \\
                                &= \begin{cases}f_{i-1}\otimes\phi_i, & j=i\\ f_{j-1}\otimes\phi_i-f_{j-i}\otimes\phi_1, & j=i+1,\ldots,m \end{cases}
        \end{nalign}   
    We have $\beta_i = J_m(0)$. So, $S_i \cong R_m$ for $i=2,\ldots,m$. Hence, $K \cong P_{m-2} \oplus (R_m)^{m-1}$.
    
    Now, we can use Corollary~\ref{cor:ker_ev_t} to extend this formula for the kernel to general $I_n$ with $\tau^{-k} I_n \cong I_1$ ($n=2k+1$):
        \begin{nalign}
            \ev[R_m][I_n] &\cong \tau^{k} \ev[\tau^{-k} R_m][\tau^{-k} I_n]\\
                          &\cong \tau^k \ev[R_m][I_1]\\
                          &\cong \tau^k(P_{m-2} \oplus (R_m)^{m-1}\\\
                          &\cong P_{m-2k-2} \oplus (R_m)^{m-1}\\
                          &\cong P_{m-n-1} \oplus (R_m)^{m-1}
        \end{nalign}
    which agrees with the desired formula and as before, for $m < n+1$, we lose the summand $P_{m-n-1}$. Together with the results of Case I, this proves the proposition.
\end{proof}

\subsection{From Preinjective to Preinjective}

\begin{customthm}{\ref{thm:ev(Im,In)}}
    Given two preinjective indecomposable representations $I_m, I_n, m,n \geq 0$, of the Kronecker quiver $\mathsf{K}$, we have that 
        \[\ker \ev[I_m][I_n] \cong (I_{m+1})^{m-n}, \quad m > n \geq 1.\]
    All other kernels are equal to 0.
\end{customthm}

\begin{proof}

Consider \(\ev[I_m][I_n]: I_m \otimes \Hom(I_m, I_n) \to I_n\). By Lemma~\ref{lem:Hom_mats}, we know that a morphism $\phi \in \Hom(I_m,I_n)$ is given by the following pairs of matrices $(A,B)$:
                \begin{nalign} 
                    A &=
                        \left[
                        \begin{array}{ccccccc}
                            a_1 & a_2 & \cdots & a_{m-n+1}    & & & \\
                                & a_1 & a_2    & \cdots & a_{m-n+1} & &  \\
                     &     & \ddots & \ddots &      & \ddots & \\
                                &     &        & a_1    & a_2  & \cdots & a_{m-n+1} \\
                        \end{array}
                        \right] \in k^{(n+1) \times (m+1)} \\
        \vspace{10pt}\\               
                    B &=
                        \left[
                        \begin{array}{ccccccc}
                            a_1 & a_2 & \cdots & a_{m-n+1} & & & \\
                                & a_1 & a_2    & \cdots & a_{m-n+1} & &  \\
                     &     & \ddots & \ddots &      & \ddots & \\
                                &     &        & a_1    & a_2  & \cdots & a_{m-n+1} \\
                        \end{array}
                        \right] \in k^{n \times m}
                \end{nalign}  
    for $m \geq n$. We will define notation for the basis of this $\Hom$-space as $\phi_i = (A_i. B_i)$ for $i=1,\ldots,m-n+1$ where $A_i, B_i$ denote the matrices $A, B$ above with $a_i=1$ and $a_k=0$ for $k \neq i$. Let
        \begin{nalign}
            I_m = &\langle e_1,\ldots,e_{m+1} \rangle \toto \langle f_1,\ldots,f_{m} \rangle\\
            I_n = &\langle g_1,\ldots,g_{n+1} \rangle \toto \langle h_1,\ldots,h_{n} \rangle 
        \end{nalign}
    Then, the maps $\phi_i$ for $i=1,\ldots,m-n+1$ can be described by
        \begin{nalign}
            \phi_i(e_j) &= \begin{cases}g_j, & j=i,\ldots,i+n\\ 0, & \text{otherwise}\end{cases}\\
            \phi_i(f_j) &= \begin{cases}h_j, & j=i,\ldots,i+n-1\\ 0, & \text{otherwise}\end{cases}
        \end{nalign}
    It is easy to see that the evaluation morphism will be surjective as $\phi_1$ will map $e_1,\ldots,e_{m+1}, f_1, \ldots,f_m$ to $g_1,\ldots,g_{n+1}, h_1, \ldots,h_n$, respectively. So, we may calculate the dimension of the kernel:
        \begin{nalign}
            \ddim K &= \ddim I_m \otimes \Hom(I_m,I_n) - \ddim I_n \\
                    &= (m-n+1)(m+1,m)-(n+1,n)\\
                    &= (m^2-mn+m+m-n+1, m^2-mn+m) - (n+1,n)\\
                    &= (m^2-mn+2m-n+1-n-1, m^2-mn+m-n)\\
                    &= (m^2-mn+2m-2n,m^2-mn+m-n)\\
                    &= (m-n)(m+2,m+1)
        \end{nalign}
    For $i=1,\ldots,m-n$, define $S_i:S_i(x) \toto S_i(y)$ to be generated at each vertex by the following sets of elements
        \begin{nalign}
            \mathcal{B}_i(x) = &\{ -e_{1} \otimes \phi_{i+1}\} \, \cup  \\
                       & \{e_{j} \otimes \phi_i - e_{j+1} \otimes \phi_{i+1} \, | \,j=1,\ldots, m\} \, \cup \\
                       & \{e_{m+1} \otimes \phi_i\} \\
                & &&\\
            \mathcal{B}_i(y) = &\{ -f_{1} \otimes \phi_{i+1}\} \, \cup  \\
                       & \{f_{j} \otimes \phi_i - f_{j+1} \otimes \phi_{i+1} \, | \,j=1,\ldots, m-1\} \, \cup \\
                       & \{f_{m} \otimes \phi_i\}         
        \end{nalign}
    where $\alpha_i, \beta_i:S_i(x) \to S_i(y)$ are given by the actions of $\alpha_{I_m},\beta_{I_m}$, respectively. 
    
    We will use Lemma~\ref{lem:recognition} to show that $K\cong \oplus_i S_i$. We begin by showing that each $S_i \subseteq K$. First, consider the basis elements at the vertex $x$. These are the elements in $S_i(x)$. 
        \[-e_1 \otimes \phi_{i+1} \mapsto \phi_{i+1}(e_1) = 0\]
    because $i=1,\ldots,m-n$ means that $i+1 \neq 1$. 
        \[e_j \otimes \phi_i - e_{j+1}\otimes\phi_{i+1} \mapsto \phi_i(e_j)-\phi_{i+1}(e_{j+1}).\]
    for $j=1,\ldots,m$. Recall that $\phi_i(e_k) = g_k$ for $k=i,\ldots,i+n$ and is otherwise 0 and note that $k=i,\ldots,i+n$ implies $k+1=i+1,\ldots,i+n+1$. Therefore, $\phi_{i+1}(e_{j+1})$ will take on the same values as $\phi_i(e_j)$. So, both terms are both $g_j$ or both $0$. In either case, their difference is $0$, as desired. Finally,
        \[e_{m+1}\otimes\phi_i \mapsto \phi_i(e_{m+1}) = 0\]
    because $i=1,\ldots,m-n$ so $i+n < m+1$. Therefore, $S_i(x) \subseteq K(x)$. The calculation to show that $S_i(y) \subseteq K(y)$ is nearly identical, so we may conclude that $S_i \subseteq K$.

    Now we provide justification that the set $\mathcal{B}(x)=\cup_i \mathcal{B}_i(x)$ and $\mathcal{B}(y)=\cup_i \mathcal{B}_i(y)$ are linearly independent. Suppose we have a linear combination of elements from $\mathcal{B}(x)$ which equals 0:
        \[0=\sum_{i=1}^{m-n} r_i(-e_1\otimes\phi_{i+1}) + \sum_{i=1}^{m-n} \sum_{j=1}^{m} s_{ij}(e_j\otimes\phi_i-e_{j+1}\otimes\phi_{i+1}) + \sum_{i=1}^{m-n} t_i(e_{m+1} \otimes \phi_i). \]

    First, recall that the sets of elements $\{e_j\}_j$ and $\{\phi_i\}_i$ are bases and thus, linearly independent. In particular, for each $k$, the terms involving $\phi_k$ must cancel each other out because we cannot write it as a linear combination of the others. Consider the terms involving $\phi_1$. We must have 

    \begin{nalign}
        0 &= \sum\limits_{j=1}^m s_{1j} e_j \otimes \phi_1 + t_1 e_{m+1} \otimes \phi_1
    \end{nalign}

    which is a linear combination of the elements $\{e_j\}_j$ and so we can conclude that $s_{1j}=0$ for all $j=1,\ldots,m$ and $t_1=0$. Similarly, we can consider $\phi_{m-n+1}$ to conclude that $s_{m-n,j}=0$ for all $j$ and $r_{m-n}=0$. Next, we move our attention to terms involving $\phi_k$ for $1<k<m-n+1$. We obtain the equation

    \begin{alignat}{4}
        0 &= -r_{k-1}e_1 \otimes \phi_k &&+ \sum\limits_{j=1}^m s_{kj} (e_j \otimes \phi_k - e_{j+1}\otimes \phi_{k+1}) \\
        & &&+ \sum\limits_{j=1}^m s_{k-1,j} (e_j \otimes \phi_{k-1} - e_{j+1}\otimes \phi_{k}) + t_k e_{m+1} \otimes \phi_k\nonumber\\
          &= [-r_{k-1}e_1 + \sum\limits_{j=1}^m &&s_{kj} e_j + \sum\limits_{j=1}^m s_{k-1,j} e_{j+1} + t_k e_{m+1}] \otimes \phi_k + \sum\limits_{j=1}^m s_{kj} e_{j+1} \otimes \phi_{k+1}\nonumber\\
          & &&+ \sum\limits_{j=1}^m s_{k-1,j} e_j \otimes \phi_{k-1}\nonumber
    \end{alignat}

    Because of the linear independence of the elements in the second argument, each term must separately equal 0. From the second and third terms, we get that $s_{kj}=0$ for $j=1,\ldots,m$. The sum then reduces to 

    \[0 = (-r_{k-1}e_1 + t_k e_{m+1}) \otimes \phi_k\]

    from which we can deduce that $r_{k-1}=t_k=0$. The argument does not depend on the choice of $k$, so we get that all of the coefficients in the sum are 0 and the elements of $\mathcal{B}(x)$ are linearly independent. A nearly identical argument works to show that $\mathcal{B}(y)$ is also a linearly independent set.

    The following shows that the dimension of each $S_i$ is $(m+2,m+1)$. 
        \begin{nalign}
            \dim S_i(x) = |\mathcal{B}_i(x)| = \, &|\{ -e_{1} \otimes \phi_{i+1}\}| \\
                            &+ |\{e_{j} \otimes \phi_i - e_{j+1} \otimes \phi_{i+1}, j=1,\ldots, m\}|\\
                            &+ |\{e_{m+1} \otimes \phi_i\}| \\
                       =  \, &1 + m + 1\\
                       =  \, &m + 2 \\
                       \\
            \dim S_i(y) = |\mathcal{B}_i(x)| =  \, &|\{ -f_{1} \otimes \phi_{i+1}\}| \\
                             &+ |\{f_{j} \otimes \phi_i - f_{j+1} \otimes \phi_{i+1}, j=1,\ldots, m-1\}|\\
                             &+ |\{f_{m} \otimes \phi_i \}| \\
                       =  \, &1 + m-1 + 1\\
                       = \, &m + 1
        \end{nalign}
    Since there are $m-n$ such representations, $\sum_i \ddim S_i = (m-n)(m+2,m+1)$, which is equal to $\ddim K$. So, by Lemma~\ref{lem:recognition}, we conclude that $K=S=\oplus_i \, S_i$. It remains to show that $S_i \cong I_{m+1}$ for each $i$. To do so, we will show that $\alpha_i = \begin{bsmallmatrix} \mathbbm{1}_{m+1} & 0\end{bsmallmatrix}, \beta_i = \begin{bsmallmatrix} 0 & \mathbbm{1}_{m+1} \end{bsmallmatrix}$. 

    We will apply $\alpha_i, \beta_i$ to the basis elements of $S_i$. Recall that $\alpha_i$ is defined by the action of $\alpha_{I_{m}}$. 
        \begin{nalign}
            \alpha_i(-e_{1} \otimes \phi_{i+1}) &=  -\alpha_{I_m}(e_{1}) \otimes \phi_{i+1}\\
                                                    &= -f_{1} \otimes \phi_{i+1}\\
        \\
            \alpha_i(e_{j} \otimes \phi_i - e_{j+1} \otimes \phi_{i+1}) &\mapsto \alpha_{I_m}(e_{j}) \otimes \phi_i - \alpha_{I_m}(e_{j+1}) \otimes \phi_{i+1}\\  
                            &= \begin{cases} f_j\otimes\phi_i-f_{j+1}\otimes\phi_{i+1}, & j=1,\ldots,m-1\\ f_m\otimes\phi_i, & j=m \end{cases} \\
        \\
            \alpha_i(e_{m+1} \otimes \phi_i) &= \alpha_{I_m}(e_{m+1}) \otimes \phi_i\\ &= 0
        \end{nalign}
    where we note that $\alpha_{I_{m}}(e_{m+1})=0$. We conclude that $\alpha_i \cong \begin{bsmallmatrix} \mathbbm{1}_{m+1} & 0 \end{bsmallmatrix}$.
        
    Recall that $\beta_i$ is defined by the action of $\beta_{I_{m}}$. So, 
        \begin{nalign}
            \beta_i(-e_{1} \otimes \phi_{i+1}) &\mapsto  -\beta_{I_m}(e_{1}) \otimes \phi_{i+1}\\
                                                    &= 0 \\
        \\
            \beta_i(e_{j} \otimes \phi_i - e_{j+1} \otimes \phi_{i+1}) &\mapsto \beta_{I_m}(e_{j}) \otimes \phi_i - \beta_{I_m}(e_{j+1}) \otimes \phi_{i+1}\\  
                            &= \begin{cases}- f_1\otimes \phi_{i+1}, & j=1\\ f_{j-1} \otimes \phi_i - f_{j} \otimes \phi_{i+1}, & j=2,\ldots,m\end{cases}\\
        \\
            \beta_i(e_{m+1} \otimes \phi_i) &\mapsto \beta_{I_m}(e_{m+1}) \otimes \phi_i\\ &= f_m\otimes\phi_i
        \end{nalign}
    where we use that $\beta_{I_{m}}$ sends the element $e_{1}$ to $0$. Thus, $\beta_i \cong \begin{bsmallmatrix} 0 & \mathbbm{1}_{m+1}\end{bsmallmatrix}$ This proves that $S_i \cong I_{m+1}$ for $i=1, \ldots, m-n$ and so, $V \cong \oplus_{i=1}^m S_i \cong (I_{m+1})^{m-n}$, as desired.
\end{proof}

\bibliographystyle{amsalpha}
\bibliography{references.bib}
\end{document}